\documentclass[11pt]{amsart}
\usepackage{graphicx}
\usepackage{amssymb}
\usepackage{amsmath}
\usepackage{amsthm}
\usepackage{amsfonts}
\usepackage{bbm}
\usepackage{tikz}
\usepackage{tikz-cd}
\usepackage{setspace,kantlipsum}
\usepackage[toc,page]{appendix}
\usepackage{hyperref}
\usepackage{xcolor}
\usepackage[allcommands]{overarrows}
\usepackage{multirow}
\usepackage{bm}
\usepackage{caption}

\definecolor{green}{RGB}{0, 204, 0}
\definecolor{forest}{RGB}{0, 118, 0}
\definecolor{orange}{RGB}{255, 170, 0}
\definecolor{blue}{RGB}{0, 64, 255}
\definecolor{purple}{RGB}{163, 0, 163}

\usepackage{float}
\usepackage{caption}

\newtheorem{thm}{Theorem}

\newtheorem{thrm}{Theorem}[section]

\newtheorem{lemma}[thrm]{Lemma}
\newtheorem{cor}[thrm]{Corollary}

\newtheorem*{quest}{Question}

\theoremstyle{definition}

\newtheorem{remark}[thrm]{Remark}
\newtheorem{defn}[thrm]{Definition}

\newtheorem*{eg}{Example}

\newcommand{\id}{\operatorname{id}}

\newcommand{\Cob}{\operatorname{\mathbf{Cob}}}
\newcommand{\BiAb}{\operatorname{\mathbf{Bi}}}

\makeatletter
\newcommand{\colim@}[2]{%
  \vtop{\m@th\ialign{##\cr
    \hfil$#1\operator@font colim$\hfil\cr
    \noalign{\nointerlineskip\kern1.5\ex@}#2\cr
    \noalign{\nointerlineskip\kern-\ex@}\cr}}%
}
\newcommand{\colim}{%
  \mathop{\mathpalette\colim@{\rightarrowfill@\scriptscriptstyle}}\nmlimits@
}
\renewcommand{\varprojlim}{%
  \mathop{\mathpalette\varlim@{\leftarrowfill@\scriptscriptstyle}}\nmlimits@
}
\renewcommand{\varinjlim}{%
  \mathop{\mathpalette\varlim@{\rightarrowfill@\scriptscriptstyle}}\nmlimits@
}
\makeatother

\newcommand{\tinytiny}{%
  \fontsize{5pt}{5pt}\selectfont
}
\newcommand{\tinyone}{%
  \fontsize{6pt}{6pt}\selectfont
}

\newcommand{\kcap}{\includegraphics[scale=.2,trim=0 .2cm 0 0]{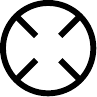}\null}
\newcommand{\kcup}{\includegraphics[scale=.15,trim=0 .2cm 0 0]{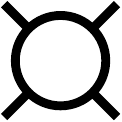}\null}

\newcommand{\kocirc}{\raisebox{-.375\height}{\includegraphics[scale=.225,trim=0 .2cm 0 0]{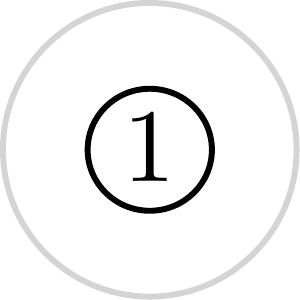}}\null}
\newcommand{\kxcirc}{\raisebox{-.375\height}{\includegraphics[scale=.225,trim=0 .2cm 0 0]{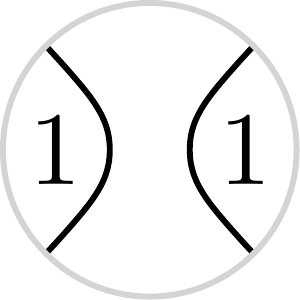}}\null}
\newcommand{\ksmooth}{\includegraphics[scale=.2,trim=0 .2cm 0 0]{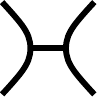}\null}

\newcommand{\kcobdo}{\raisebox{-.375\height}{\includegraphics[scale=.225,trim=0 .2cm 0 0]{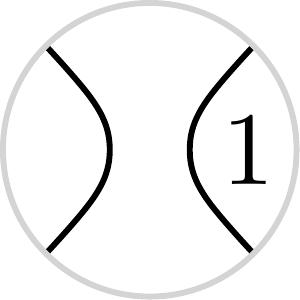}}\null}
\newcommand{\kcobdx}{\raisebox{-.375\height}{\includegraphics[scale=.225,trim=0 .2cm 0 0]{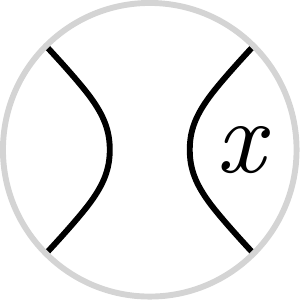}}\null}
\newcommand{\kcobdoo}{\raisebox{-.375\height}{\includegraphics[scale=.225,trim=0 .2cm 0 0]{plots/cobsmooth1oo.pdf}}\null}
\newcommand{\kcobdox}{\raisebox{-.375\height}{\includegraphics[scale=.225,trim=0 .2cm 0 0]{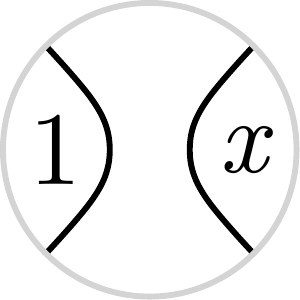}}\null}
\newcommand{\kcobdxo}{\raisebox{-.375\height}{\includegraphics[scale=.225,trim=0 .2cm 0 0]{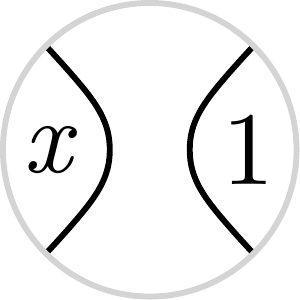}}\null}
\newcommand{\kcobdxx}{\raisebox{-.375\height}{\includegraphics[scale=.225,trim=0 .2cm 0 0]{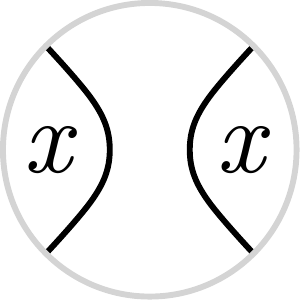}}\null}

\newcommand{\kcobio}{\raisebox{-.375\height}{\includegraphics[scale=.225,trim=0 .2cm 0 0]{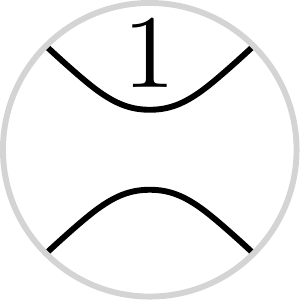}}\null}
\newcommand{\kcobix}{\raisebox{-.375\height}{\includegraphics[scale=.225,trim=0 .2cm 0 0]{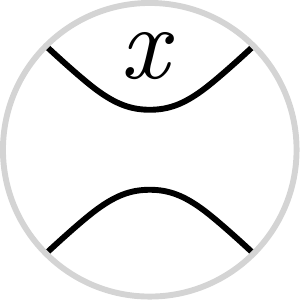}}\null}

\newcommand{\kcobiox}{\raisebox{-.375\height}{\includegraphics[scale=.225,trim=0 .2cm 0 0]{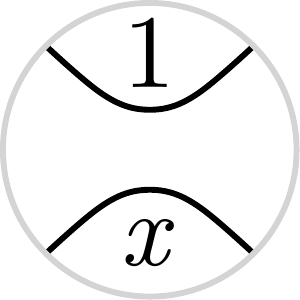}}\null}
\newcommand{\kcobixo}{\raisebox{-.375\height}{\includegraphics[scale=.225,trim=0 .2cm 0 0]{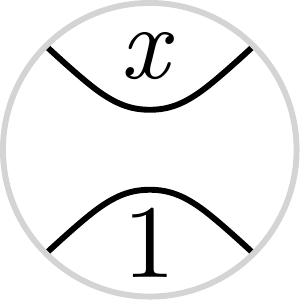}}\null}
\newcommand{\kcobixx}{\raisebox{-.375\height}{\includegraphics[scale=.225,trim=0 .2cm 0 0]{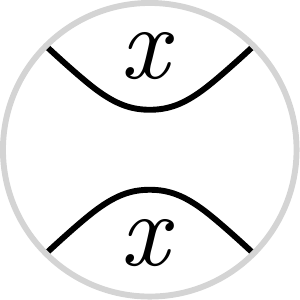}}\null}

\newcommand{\kropos}{\raisebox{-.375\height}{\includegraphics[scale=.225,trim=0 .2cm 0 0]{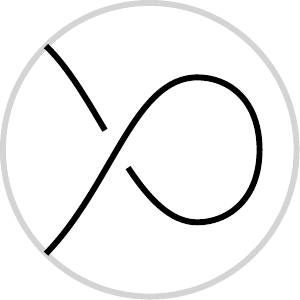}}\null}
\newcommand{\kroarc}{\raisebox{-.375\height}{\includegraphics[scale=.225,trim=0 .2cm 0 0]{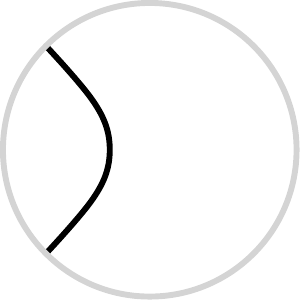}}\null}
\newcommand{\krosa}{\raisebox{-.375\height}{\includegraphics[scale=.225,trim=0 .2cm 0 0]{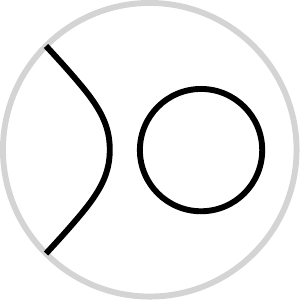}}\null}
\newcommand{\krosb}{\raisebox{-.375\height}{\includegraphics[scale=.225,trim=0 .2cm 0 0]{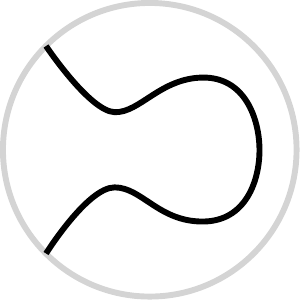}}\null}
\newcommand{\kroma}{\raisebox{-.375\height}{\includegraphics[scale=.225,trim=0 .2cm 0 0]{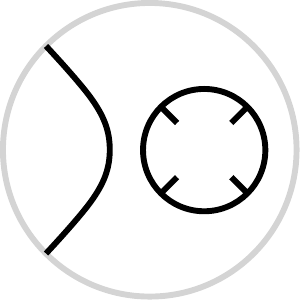}}\null}

\newcommand{\krtcr}{\raisebox{-.375\height}{\includegraphics[scale=.225,trim=0 .2cm 0 0]{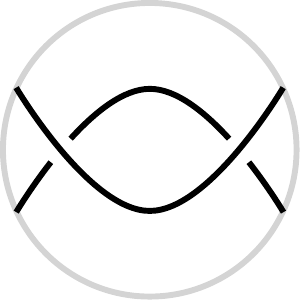}}\null}
\newcommand{\krtcrl}{\raisebox{-.375\height}{\includegraphics[scale=.225,trim=0 .2cm 0 0]{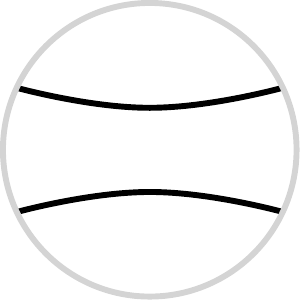}}\null}
\newcommand{\krtsa}{\raisebox{-.375\height}{\includegraphics[scale=.225,trim=0 .2cm 0 0]{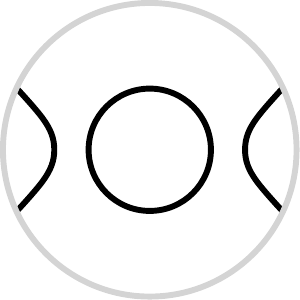}}\null}
\newcommand{\krtsb}{\raisebox{-.375\height}{\includegraphics[scale=.225,trim=0 .2cm 0 0]{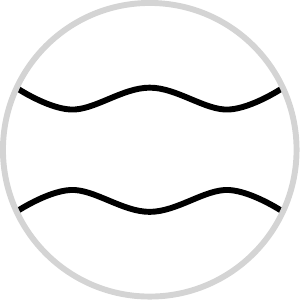}}\null}
\newcommand{\krtsc}{\raisebox{-.375\height}{\includegraphics[scale=.225,trim=0 .2cm 0 0]{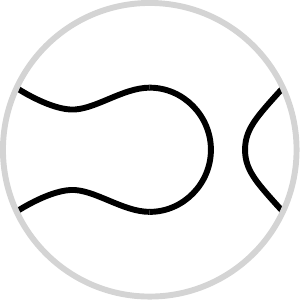}}\null}
\newcommand{\krtsd}{\raisebox{-.375\height}{\includegraphics[scale=.225,trim=0 .2cm 0 0]{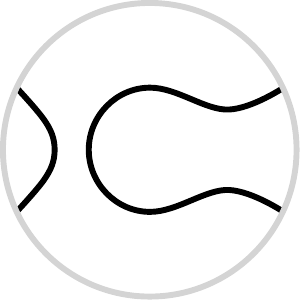}}\null}
\newcommand{\krtma}{\raisebox{-.375\height}{\includegraphics[scale=.225,trim=0 .2cm 0 0]{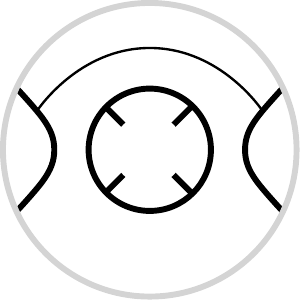}}\null}

\usepackage[margin=2.295cm]{geometry}

\title{End Khovanov homology and exotic Lagrangian planes}
\author{Yikai Teng}

\begin{document}
\begin{spacing}{1}

\begin{abstract}
    In this paper, we define the end Khovanov homology, which is an invariant of properly embedded surfaces in $\mathbb{R}^4$ up to ambient diffeomorphism. Moreover, we apply this invariant to detect the first known examples of exotic Lagrangian and symplectic planes in $\mathbb{R}^4$. 
\end{abstract}

\maketitle

\section{Introduction}

    A properly smoothly embedded plane in $\mathbb{R}^4$ is said to be an \textit{exotic plane} if it is topologically but not smoothly isotopic to the standard $xy$-plane. The first known examples of exotic planes were given by Gompf in \cite{Gom84} and by Freedman in unpublished notes.

    On the other hand, if we restrict the embedding of the plane to be algebraic, this exotic knotting behavior vanishes: it is shown in \cite{AM75} and \cite{Rud82} that any algebraically embedded plane in $\mathbb{C}^2$ is algebraically isomorphic, and hence smoothly isotopic, to the standard embedding $\mathbb{C}\hookrightarrow \mathbb{C}^2$. However, it remains unclear which additional geometric structure imposed on the smoothly embedded plane is responsible for the observed rigidity. In particular, Gompf asked in \cite{Gom25} (Problem 6.14) whether there exists any exotic Lagrangian, symplectic, or holomorphic plane in $\mathbb{R}^4\cong \mathbb{C}^2$. 
    
    In this paper, we give a positive answer to the Lagrangian and symplectic parts of this question.

    \begin{thm}{\label{thrm-exist}}
        There exists a proper embedding of a plane in $\mathbb{R}^4$ that is topologically but not smoothly isotopic to the standardly embedded plane. Moreover, this surface can be smoothly isotoped to a Lagrangian submanifold of $\mathbb{R}^4$ equipped with the standard symplectic structure.
    \end{thm}

    One such example is the surface $\Sigma$ shown in Figure \ref{20250726-1}.

    \begin{figure}[H]
        \centering
        \begin{picture}(17cm,3.7cm)
            \put(0cm,-0.1cm){\includegraphics[width=17cm, height=4cm]{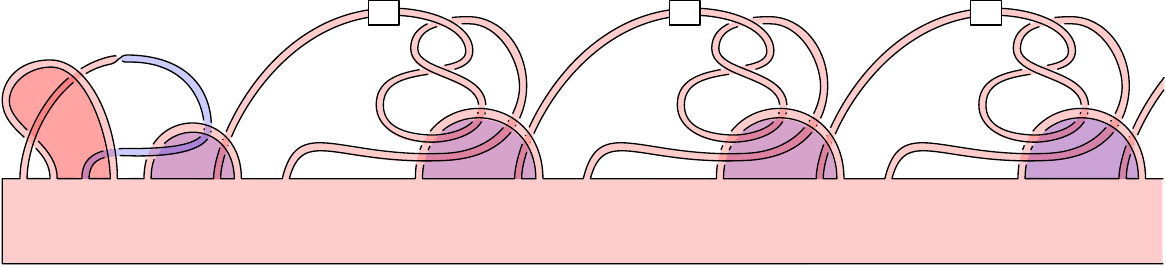}}
                                
            \put(5.4cm, 3.605cm){\scriptsize{$-2$}}
            \put(9.78cm, 3.605cm){\scriptsize{$-2$}}
            \put(14.18cm, 3.605cm){\scriptsize{$-2$}}
                
        \end{picture}        
        \caption{An exotic plane $\Sigma$ that can be smoothly isotoped to be Lagrangian.\vspace{-5pt}}
        \label{20250726-1}
    \end{figure}
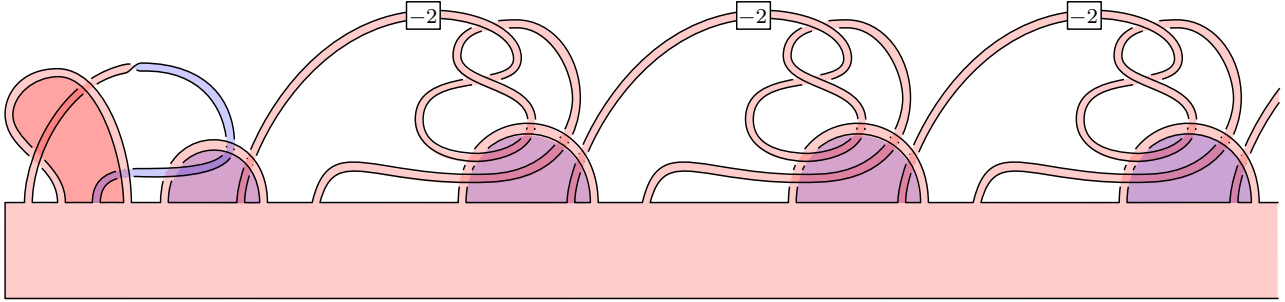

    The surface $\Sigma$ can also be smoothly isotoped to be a symplectic submanifold of $(\mathbb{R}^4,\omega_{\text{std}})$; we discuss this briefly at the end of this paper. The holomorphic counterpart for this question remains an open problem (cf. \cite{Hay21c}).

    \begin{quest}
        Does there exist an exotic plane properly embedded in $\mathbb{C}^2$ that is holomorphic?
    \end{quest}

    The majority of this paper will be devoted to showing that the surface $\Sigma$ is not smoothly isotopic to the standard plane. To the author's knowledge, there are only two known methods that detect exotic planes: via the double cover of $\mathbb{R}^4$ branched along the surface, and via the minimal genus at infinity $g^\infty$ and the kinkiness at infinity $\kappa_\pm^\infty$ developed in \cite{Gom25}. In the same paper, Gompf asks (Problem 6.13) whether there is any combinatorial invariant that detects noncompact exotic surfaces properly embedded in $\mathbb{R}^4$. 
    
    In this paper, we define the \textit{co-end Khovanov homology} $\overrightharpoonup{Kh}$ and \textit{end Khovanov homology} $\overleftharpoonup{Kh}$ with $\mathbb{Z}/2$ coefficients, both of which are combinatorial invariants of surfaces embedded in $\mathbb{R}^4$ up to ambient diffeomorphisms. Our constructions are parallel to that of the end Floer homology, defined in \cite{Gad10}. We sketch the constructions of (co)end Khovanov homologies below.

    Let $\Sigma$ be a noncompact surface with finite Euler characteristic (and empty boundary) properly and smoothly embedded in $\mathbb{R}^4$. Consider a compact exhaustion of $\mathbb{R}^4$ by 4-balls $B_0\subset B_1\subset\cdots$ such that $B_0\cap \Sigma=\varnothing$. Assuming that the 3-spheres $S_i=\partial B_i$ are transverse to the surface $\Sigma$, we obtain a sequence of links $L_i=\Sigma\cap S_i$. The (co)end Khovanov homologies are defined as the (co)limits of the groups $Kh(L_i)$ under the maps induced by the cobordisms $C_i=\Sigma\cap (B_i\backslash\mathring{B}_{i-1})$ between $L_{i-1}$ and $L_i$. Moreover, the (co)end Khovanov homologies split in terms of an integral bigrading $(h,q)$.

    In this paper, we prove that the co-end Khovanov homology $\overrightharpoonup{Kh}^{h,q}$ and the end Khovanov homology $\overleftharpoonup{Kh}^{h,q}$ are invariants up to ambient diffeomorphisms. In particular, we prove the following theorem.

    \begin{thm}{\label{isotopy-inv}}
        Let $\Sigma$ and $\Sigma'$ be two properly embedded surfaces in $\mathbb{R}^4$ with finite Euler characteristics. If there is a diffeomorphism $f:(\mathbb{R}^4, \Sigma)\rightarrow(\mathbb{R}^4, \Sigma')$ between the two surfaces, then, for any bigrading $(h,q)$, there are induced isomorphisms $f_\ast: \overrightharpoonup{Kh}^{h,q}(\Sigma)\rightarrow\overrightharpoonup{Kh}^{h,q}(\Sigma')$ and $f_\ast: \overleftharpoonup{Kh}^{h,q}(\Sigma)\rightarrow\overleftharpoonup{Kh}^{h,q}(\Sigma')$ on (co)end Khovanov homologies.
    \end{thm}

    In particular, we show that the end Khovanov homology of the surface described in Figure \ref{20250726-1} is not isomorphic to that of the standardly embedded plane. Hence, it is not smoothly isotopic to the standard plane.

    However, it is shown in \cite{Gom25} that there exist uncountably many exotic planes in $\mathbb{R}^4.$ It remains a question whether (co)end Khovanov homologies can detect finer differences between exotic planes.

    \begin{quest}
        Can (co)end Khovanov homologies detect different exotic planes?
    \end{quest}

    Given the parallels with Gadgil's construction in \cite{Gad10}, one can ask about connections with Heegaard Floer homology. In \cite{OS05}, Ozsv\'ath-Szab\'o established a spectral sequence relating the reduced Khovanov homology $\widetilde{Kh}(L)$ of a link $L\subset S^3$ to the (hat-flavored) Heegaard Floer homology of $\Sigma_2(-L)$, the double cover of $S^3$ branched along the mirror of $L$. We ask whether such a relationship extends to our setting.

    \begin{quest}
        Let $\Sigma\subset \mathbb{R}^4$ be a noncompact, properly embedded surface with finite Euler characteristic. Is there a spectral sequence whose $E^2$-page is (a reduced version of) co-end Khovanov homology and which converges to the end Floer homology of the double cover of $\mathbb{R}^4$ branched along $\Sigma$?
    \end{quest}

    We leave these and other questions to be addressed in future work.

    \subsection{Organization of the paper}
        In Section \ref{Kho-defn}, we review the definition of Khovanov homology and the related TQFT. In Section \ref{sec-end}, we define the end Khovanov homologies, and prove Theorem \ref{isotopy-inv}. In Section \ref{exotic-construction}, we construct the surface shown in Figure \ref{20250726-1}, and show that it is indeed an exotic plane using end Khovanov homology. In Section \ref{sec-lagrangian}, we show that the surface $\Sigma$ can be isotoped to a Lagrangian (or symplectic) embedding. Finally, in the appendix, we review Khovanov homology for links in manifolds abstractly diffeomorphic to $\mathbb{R}^3$ and $S^3$, which is necessary to justify the statements in Section \ref{sec-end} rigorously.

    \subsection{Acknowledgments}
        The author is grateful to Kyle Hayden for suggesting this problem and for his ongoing guidance, and to Marco Golla for conversations that helped shape this work.

\section{Background on Khovanov homology}{\label{Kho-defn}}

    Khovanov homology is a link homology theory first introduced for links in $\mathbb{R}^3$ in \cite{Kho00}. Its functoriality in $\mathbb{R}^3\times I$ was later developed in \cite{Jac04}, making Khovanov theory a TQFT. In \cite{MWW22}, the Khovanov TQFT was extended to links in $S^3$ and link cobordisms in $S^3\times I$. In this section, we review the Khovanov homology TQFT following the paper \cite{BN05}. We will apply this TQFT for our calculations in Sections \ref{sec-end} and \ref{exotic-construction}. To avoid sign issues, throughout this paper, we assume that the base ring for all Khovanov homology is $\mathbb{Z}/2.$ 

    Given a link diagram $D$ with $n$ crossings, we denote the number of positive crossings by $n_+$, and the number of negative crossings $n_-$. We also fix an enumeration for the $n$ crossings. A \textit{smoothing} $\sigma$ of $D$ is a 1-manifold associated to the diagram by resolving each of its crossings in one of the two ways indicated in Figure \ref{resolutions}. Note that a smoothing $\sigma$ can be represented as a binary sequence $(\sigma_1,...,\sigma_n)$, determined by the type of resolution at each crossing. A \textit{labeled smoothing} $\alpha_\sigma$ is a smoothing together with a choice of labeling, either $x$ or $1$, for each connected component of the smoothing $\sigma$.

        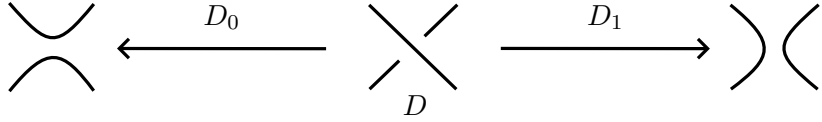
\begin{figure}
            \centering
            \begin{picture}(10.9cm,1.5cm)
                \put(0cm,0.2cm){\includegraphics[height=1.2cm]{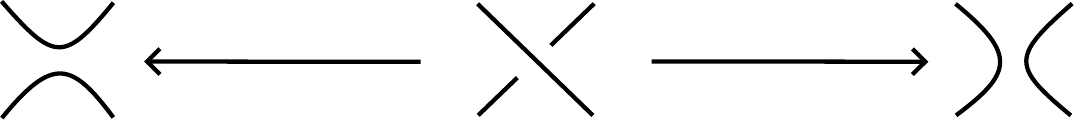}}
                \put(7.7cm,1.1cm){$D_1$}
                \put(2.6cm,1.1cm){$D_0$}
                \put(5.25cm,-0.1cm){{$D$}}
            \end{picture}        
            \caption{0- and 1- resolutions of a crossing. \vspace{-9pt}}
            \label{resolutions}
        \end{figure}  

    The \textit{Khovanov chain group} $CKh(D)$ is the vector space generated by all possible labeled smoothings for the diagram $D$. It is equipped with a bigrading: the homological grading $h$, and the quantum grading $q$. Fixing a labeled smoothing $\alpha_\sigma$, let $|\sigma|$ denote the number of 1-smoothings in $\sigma$, and let $v_+$ and $v_-$ denote the number of $1$-labels and $x$-labels in $\alpha_\sigma$. Then $h$ and $q$ are defined to be 
    $$h:= |\sigma| - n_-$$
    $$q:= v_+ - v_- + h + n_+ - n_-.$$

    The Khovanov chain group also comes with a differential $d$, making $CKh(D)$ a chain complex. Given a labeled smoothing $\alpha_\sigma$ such that $\sigma_i=0$, let $\sigma^i$ denote the same smoothing except that we replace the $i^{\text{th}}$ 0-smoothing by a 1-smoothing, and let $\alpha_{\sigma^i}$ be the labeled smoothing obtained by applying the saddle induced chain map from Table \ref{table_morse}. Then the differential is defined by
    $$d(\alpha_\sigma):= \sum_{\{i\,\mid\,\sigma_i = 0\}} \alpha_{\sigma^i}.$$

    The \textit{Khovanov homology} $Kh(D)$ is defined to be the homology of the Khovanov chain complex. Naturally, it inherits a bigrading. Note that the differential $d$ raises the homological grading $h$ by 1 and preserves the quantum grading $q$.

    Before we move on to chain maps induced by link cobordisms, we state the following theorem, which is useful to check whether a labeled smoothing is a cycle (hence representing a homology class).

    \begin{lemma}[Proposition 3.2 of \cite{Ell10}]{\label{when-cycle}}
        A labeled smoothing $\alpha_\sigma$ is a cycle if and only if every 0-smoothing, when changed into a 1-smoothing, merges two $x$-labeled components.
    \end{lemma}

    Next we review the Khovanov chain map induced by link cobordisms. Given a link cobordism $\Sigma\subset S^3\times I$ between link diagrams $D_0$ and $D_1$, there is a natural bigraded chain map 
    $$CKh(\Sigma): CKh^{h,q}(D_0)\rightarrow CKh^{h,q+\chi(\Sigma)}(D_1)$$
    that induces a homology level homomorphism $Kh(\Sigma): Kh^{h,q}(D_0)\rightarrow Kh^{h,q+\chi(\Sigma)}(D_1)$. In \cite{Jac04}, Jacobsson shows that this map is actually an isotopy invariant.

    A cobordism between link diagrams can be expressed as a finite sequence of elementary cobordisms: three Reidemeister moves, and three Morse moves. Thanks to the functoriality of Khovanov theory, in order to construct the induced map $Kh(\Sigma): Kh(L_0)\rightarrow Kh(L_1)$, we only need to define the induced maps for the elementary cobordisms. The relevant constructions are shown in Tables \ref{table_morse} and \ref{table_reidemeister_redux} below. For a full list of moves, we refer the readers to \cite{BN05}.

    \begin{table}
        \center
        \renewcommand{\arraystretch}{2.25}
        \begin{tabular}{|c|c|c|l|}
            \hline
            \text{Morse Move} & \text{Ornament} & \text{Chain map} &  \multicolumn{1}{c|}{\text{Definition of chain map}} \\
            \hline
            birth & \kcup & $\iota$ & $\begin{array}{l} \hspace{1.25em} 1 \hspace{1.3em} \mapsto \kocirc \vspace{.425em} \end{array}$ \\
            \hline
            death & \kcap & $\varepsilon$ & $\begin{array}{l} \kocirc \mapsto \hspace{1.25em} 0 \\ \kxcirc \mapsto \hspace{1.25em} 1 \vspace{.425em} \end{array}$ \\
            \hline
            \multirow{2}{80pt}{\hspace{2.15em}\vspace{-5em}saddle} & $\begin{array}{l} \vspace{-6em}\ksmooth \end{array}$ & $m$ & $\begin{array}{l} \kcobdoo \mapsto \kcobio \\ \kcobdox \mapsto \kcobix \\ \kcobdxo \mapsto \kcobix \\ \kcobdxx \mapsto \hspace{1.25em} 0 \vspace{.425em} \end{array}$ \\\cline{3-4} & & $\Delta$ & $\begin{array}{l} \kcobdo \mapsto \kcobiox + \kcobixo \\ \kcobdx \mapsto \kcobixx \vspace{.425em} \end{array}$ \\
            \hline
        \end{tabular}
        \renewcommand{\arraystretch}{1}

        \caption{The chain maps induced by Morse moves. This table is from \cite{HS24}.\vspace{5pt}}
        \label{table_morse}
    \end{table}

    \begin{table}\center
	\renewcommand{\arraystretch}{2.25}
	\begin{tabular}{|c|c|c|}
		\hline
		Reidemeister move & Smoothing & Induced chain map \\
		\hline
		\multirow{2}{80pt}{\begin{minipage}{9em}
				$\kropos \to \kroarc$
			\end{minipage}} & \krosa & \kroma $\begin{array}{c}\hspace{-1em}\vspace{.425em}\end{array}$ \\\cline{2-3} & \krosb & $\begin{array}{c} 0 \vspace{.425em} \end{array}$ \\
		\hline
		\multirow{2}{80pt}{\begin{minipage}{9em}
				$\krtcr \to \krtcrl$
		\end{minipage}} & \krtsa & $-$ \!\! \krtma $\begin{array}{c}\hspace{.25em}\vspace{.425em}\end{array}$ \\\cline{2-3} & \krtsb & \krtcrl $\begin{array}{c}\hspace{-1em}\vspace{.425em}\end{array}$ \\\cline{2-3} & \krtsc & $\begin{array}{c} 0 \vspace{.425em} \end{array}$ \\\cline{2-3} & \krtsd & $\begin{array}{c} 0 \vspace{.425em} \end{array}$ \\
		\hline
	\end{tabular}
	\renewcommand{\arraystretch}{1}

        \captionsetup{width = 13.2cm}
        \caption{The relevant chain maps induced by Reidemeister I and II moves. This table is adapted from \cite{HS24}.}
	\label{table_reidemeister_redux}
    \end{table}
            
        \subsection{Abstract Khovanov homology}{\label{review-abs}}

            In \cite{MWW22}, Morrison-Walker-Wedrich generalized the Khovanov TQFT to links in 3-manifolds abstractly diffeomorphic to $S^3$ and link cobordisms in 4-manifolds abstractly diffeomorphic to $S^3\times I$.

            Given a link $L$ in a 3-manifold $S$ abstractly diffeomorphic to $S^3$, we denote its abstract Khovanov homology by $Kh(S,L)$. Given a link cobordism $(W,\Sigma): (S_0,L_0)\rightarrow (S_1,L_1)$ where $W$ is abstractly diffeomorphic to $S^3\times I$, we denote the associated cobordism map by $Kh(W,\Sigma): Kh(S_0,S_1)\rightarrow Kh(S_1,L_1)$. To establish the invariance of the end Khovanov homologies later in Section \ref{sec-end}, we rely on the naturality of this theory under ambient diffeomorphisms. In particular, we have the following key lemma.

            \begin{lemma}{\label{iso-natural}}

                Let $(W,\Sigma): (S_0,L_0)\rightarrow (S_1,L_1)$ and $(W',\Sigma'): (S'_0,L'_0)\rightarrow (S'_1,L'_1)$ be link cobordisms where $W,W'$ are abstractly diffeomorphic to $S^3\times I$. Then any diffeomorphism of pairs $f:(W,\Sigma)\rightarrow (W',\Sigma')$ induces bigrading-preserving natural isomorphisms $f_\ast:Kh^{h,q}(S_0,L_0)\rightarrow Kh^{h,q}(S'_0,L'_0)$ and $f_\ast:Kh^{h,q}(S_1,L_1)\rightarrow Kh^{h,q}(S'_1,L'_1)$ for each bigrading $(h,q)$. Namely, the following diagram commutes:
                \begin{center}
                    \begin{tikzcd}
                        {Kh^{h,q}(S_0,L_0)} \arrow[d, "f_\ast"] \arrow[rr, "{Kh(W,\Sigma)}"] &  & {Kh^{h,q+\chi(\Sigma)}(S_1,L_1)} \arrow[d, "f_\ast"] \\
                        {Kh^{h,q}(S'_0,L'_0)} \arrow[rr, "{Kh(W',\Sigma')}"]                 &  & {Kh^{h,q+\chi(\Sigma')}(S'_1,L'_1)}                   
                    \end{tikzcd}
                \end{center}
            \end{lemma}

            For a more detailed exposition of abstract Khovanov homology, we refer the readers to the appendix.                

    \section{End Khovanov Homologies}{\label{sec-end}}

        Let $\Sigma$ be a noncompact surface (with empty boundary) properly and smoothly embedded in $\mathbb{R}^4$. Further assume that $\Sigma$ has finite Euler characteristic. In this section we construct the co-end Khovanov homology $\overrightharpoonup{Kh}^{h,q}(\Sigma)$ and the end Khovanov homology $\overleftharpoonup{Kh}^{h,q}(\Sigma)$ for the surface $\Sigma$.
        
        Let $B_0 \subset B_1 \subset B_2\subset\cdots$ be a compact exhaustion of $\mathbb{R}^4$ by 4-balls $B_i$ such that $B_0\cap \Sigma=\varnothing$. Denote the 3-sphere $\partial B_i$ by $S_i$ and the surface $B_i\cap \Sigma$ by $\Sigma_i$. After perturbing the surface $\Sigma$ slightly so that it is transverse to the spheres $S_i$, we obtain a sequence of links $L_i = S_i\cap \Sigma$. In this sense $\varnothing = \Sigma_0\subset \Sigma_1\subset  \Sigma_2\subset \cdots$ is a compact exhaustion for the surface $\Sigma$, and each $C_i:=\Sigma_i\backslash\mathring{\Sigma}_{i-1}$ is a cobordism from $L_{i-1}$ to $L_i$ in the 4-manifold $W_i:=B_i\backslash \mathring{B}_{i-1}$, which is abstractly isomorphic to $S^3\times I$. A schematic is shown in Figure \ref{schematics}.

            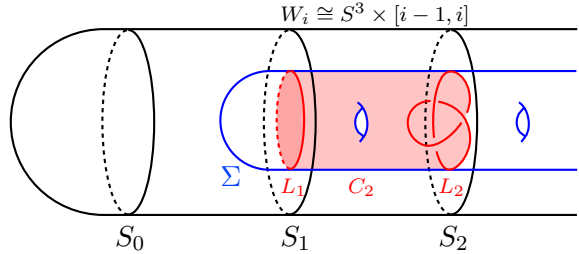
\begin{figure}[H]
                \centering
                \begin{picture}(7cm,2.8cm)
                    \put(0cm,0.3cm){\includegraphics[height=2.5cm]{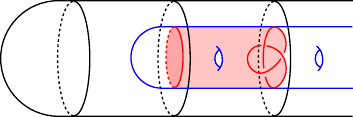}}
                    \put(1.4cm,-0.1cm){$S_0$}
                    \put(3.6cm,-0.1cm){$S_1$}
                    \put(5.7cm,-0.1cm){$S_2$}

                    \put(2.8cm,0.7cm){\textcolor{blue}{$\Sigma$}}
                    \put(3.6cm,0.6cm){\textcolor{red}{\scriptsize{$L_1$}}}
                    \put(5.7cm,0.6cm){\textcolor{red}{\scriptsize{$L_2$}}}
                    \put(4.5cm,0.6cm){\textcolor{red}{\scriptsize{$C_2$}}}

                    \put(3.6cm,2.9cm){\scriptsize{$W_i\cong S^3\times [i-1,i]$}}
                \end{picture}        
                \caption{A schematic for the surface $\Sigma\subset \mathbb{R}^4$.}
                \label{schematics}
            \end{figure} 

        Taking the Khovanov homology, we have a direct system:
        $$Kh(\varnothing)\cong Kh(L_0) \xrightarrow{Kh(W_1,C_1)} Kh(L_1) \xrightarrow{Kh(W_2,C_2)} Kh(L_2)\xrightarrow{Kh(W_3,C_3)} \cdots$$
        We define the \textit{co-end Khovanov homology} $\overrightharpoonup{Kh}(\Sigma)$ to be the direct limit of this direct system.

        Next we describe the bigrading of the co-end Khovanov homology inherited from regular Khovanov homology. As the Euler characteristic of the surface $\Sigma$ is assumed to be finite, we can make a degree shift (depending on the Euler characteristic of the surface $\Sigma\backslash\mathring{\Sigma}_i$), so that all the cobordism-induced-maps $Kh(W_i,C_i)$ in the shifted direct system preserve the bigrading of the original Khovanov homology. Under this degree shift, the direct system splits, for each bigrading $(h,q)$, into the direct sum of the following direct systems:
        $$Kh^{h,q-\chi(\Sigma\backslash \mathring{\Sigma}_0)}(L_0) \xrightarrow{Kh(W_1,C_1)} Kh^{h,q-\chi(\Sigma\backslash \mathring{\Sigma}_1)}(L_1) \xrightarrow{Kh(W_2,C_2)} Kh^{h,q-\chi(\Sigma\backslash \mathring{\Sigma}_2)}(L_2)\xrightarrow{Kh(W_3,C_3)}\cdots $$        
        We define $\overrightharpoonup{Kh}^{h,q}(\Sigma)$ to be the colimit of this direct system, and it follows that $\overrightharpoonup{Kh}(\Sigma)\cong \bigoplus_{h,q} \overrightharpoonup{Kh}^{h,q}(\Sigma)$.
        
        We can also consider the opposite link cobordism $(\overline{W}_i,\overline{C}_i): (S_i,L_i)\rightarrow(S_{i-1},L_{i-1})$. Again by taking the Khovanov homology, we have the inverse system:
        $$Kh(\varnothing)\cong Kh(L_0) \xleftarrow{Kh(\overline{W}_1,\overline{C}_1)} Kh(L_1) \xleftarrow{Kh(\overline{W}_2,\overline{C}_2)} Kh(L_2)\xleftarrow{Kh(\overline{W}_3,\overline{C}_3)} \cdots$$

        We define the \textit{end Khovanov homology} $\overleftharpoonup{Kh}(\Sigma)$ as the inverse limit of this inverse system. Also similarly to the previous case, after a degree shift determined by the Euler characteristic of $\Sigma\backslash\mathring{\Sigma}_i$, the inverse system above splits, for each pair $(h,q)$, into the direct product of the following inverse systems:
        $$Kh^{h,q+\chi(\Sigma\backslash \mathring{\Sigma}_0)}(L_0) \xleftarrow{Kh(\overline{W}_1,\overline{C}_1)} Kh^{h,q+\chi(\Sigma\backslash \mathring{\Sigma}_1)}(L_1) \xleftarrow{Kh(\overline{W}_2,\overline{C}_2)} Kh^{h,q+\chi(\Sigma\backslash \mathring{\Sigma}_2)}(L_2)\xleftarrow{Kh(\overline{W}_3,\overline{C}_3)}\cdots $$
        We define $\overleftharpoonup{Kh}^{h,q}(\Sigma)$ to be the limit of this inverse system, and it follows that $\overleftharpoonup{Kh}(\Sigma)\cong \prod_{h,q} \overleftharpoonup{Kh}^{h,q}(\Sigma).$ 

        \begin{remark}
            Note that the end Khovanov homology $\overleftharpoonup{Kh}(\Sigma)$ is the direct product (instead of direct sum) of each $\overleftharpoonup{Kh}^{h,q}(\Sigma)$. This means that the end Khovanov homology $\overleftharpoonup{Kh}(\Sigma)$ is, in general, not a bigraded vector space, since $\overleftharpoonup{Kh}^{h,q}(\Sigma)$ can be nontrivial for infinitely many pairs of $(h,q)$. However, this difference in algebraic structures does not play an important role in this paper. For simplicity, we will still refer to $\overleftharpoonup{Kh}(\Sigma)$ as bigraded vector spaces, the pair $(h,q)$ as bigradings, and $\overleftharpoonup{Kh}^{h,q}(\Sigma)$ as homogeneous spaces.
        \end{remark}

        \begin{remark}
            If the surface $\Sigma$ has infinite Euler characteristic, then the $q$-grading shift defined above is no longer applicable. However, the (co-)end Khovanov homologies still split with respect to the homological grading $h$, and thus can still be thought of as a (singly) graded vector space over $\mathbb{Z}/2$.
        \end{remark}

        \begin{lemma}
            The co-end Khovanov homology $\overrightharpoonup{Kh}^{h,q}(\Sigma)$ and end Khovanov homology $\overleftharpoonup{Kh}^{h,q}(\Sigma)$ are independent of the chosen exhaustion.
        \end{lemma}

        \begin{proof}
            This proof is analogous to the proof of Proposition 3.4 of \cite{Gad10}. We first note that by the definition of the direct (or inverse) limit, the limit does not change when passing to a subsequence of the direct (resp. inverse) system. 

            Consider two compact exhaustions for the pair $(\mathbb{R}^4,\Sigma)$: $$(B_0,\Sigma_0)\subset (B_1,\Sigma_1)\subset (B_2,\Sigma_2)\subset \cdots$$ and $$(B'_0,\Sigma'_0)\subset (B'_1,\Sigma'_1)\subset (B'_2,\Sigma'_2)\subset \cdots$$
            
            We can assume, by passing to subsequences, that $$(B_0,\Sigma_0)\subset (B'_0,\Sigma'_0)\subset (B_1,\Sigma_1)\subset (B'_1,\Sigma'_1)\subset\cdots$$ However, the two original exhaustions are both sub-exhaustions of the combined exhaustion, which indicates that all three exhaustions give isomorphic direct (resp. inverse) limits.
        \end{proof}

        Next we verify that the end and co-end Khovanov homologies are actually invariants up to ambient diffeomorphisms. In particular, we establish Theorem \ref{isotopy-inv}, which asserts that a diffeomorphism $f:(\mathbb{R}^4, \Sigma)\rightarrow(\mathbb{R}^4, \Sigma')$ between two embedded surfaces with finite Euler characteristics induces homogeneous isomorphisms $f_\ast: \overrightharpoonup{Kh}(\Sigma)\rightarrow\overrightharpoonup{Kh}(\Sigma')$ and $f_\ast: \overleftharpoonup{Kh}(\Sigma)\rightarrow\overleftharpoonup{Kh}(\Sigma')$.

        \begin{proof}[Proof of Theorem \ref{isotopy-inv}]
            We only prove the co-end Khovanov homology case. The end version is similar.

            By Lemma \ref{iso-natural}, the diffeomorphism $f$, after restrictions, induces isomorphisms 
            $$f_\ast: \overrightharpoonup{Kh}^{h,q-\chi(\Sigma\backslash\mathring{\Sigma}_i)}(S_i,L_i)\rightarrow \overrightharpoonup{Kh}^{h,q-\chi(\Sigma'\backslash\mathring{\Sigma}'_i)}(S'_i, L'_i).$$
            
            Note that $\Sigma\backslash\mathring{\Sigma}_i$ and $\Sigma'\backslash\mathring{\Sigma}'_i$ have the same Euler characteristic (which we will denote by $\chi_i$ for convenience), so $f_\ast$ is bigrading-preserving as expected. Now Lemma \ref{iso-natural} ensures that the following diagram commutes.
            \begin{center}
                \vspace{-10pt}
                \begin{tikzcd}
                    {Kh^{h,q-\chi_0}(S_0,L_0)} \arrow[d, "f_\ast"] \arrow[rr, "{Kh(W_1,C_1)}"] &  & {Kh^{h,q-\chi_1}(S_1,L_1)} \arrow[d, "f_\ast"] \arrow[rr, "{Kh(W_2,C_2)}"] &  & {Kh^{h,q-\chi_2}(S_2,L_2)} \arrow[d, "f_\ast"] \arrow[rr, "{Kh(W_3,C_3)}"] &  & \cdots \\
                    {Kh^{h,q-\chi_0}(S'_0,L'_0)} \arrow[rr, "{Kh(W'_1,C'_1)}"]                &  & {Kh^{h,q-\chi_1}(S'_1,L'_1)} \arrow[rr, "{Kh(W'_2,C'_2)}"]                &  & {Kh^{h,q-\chi_2}(S'_2,L'_2)} \arrow[rr, "{Kh(W'_3,C'_3)}"]               &  & \cdots
                \end{tikzcd}
            \end{center}

            As $f$ induces a natural isomorphism between the two direct systems, it follows that $f$ induces an isomorphism between their direct limits.
        \end{proof}

        \begin{remark}
            Although we need abstract Khovanov homology to establish the smooth invariance of the (co)end Khovanov homology, in practice we usually use round 4-balls in the compact exhaustion for $\mathbb{R}^4.$ In the following sections of this paper, we will carry out all Khovanov homology calculations using only standard Khovanov theory in the standard $S^3$ (cf. Section \ref{Kho-defn}).
        \end{remark}

        \begin{eg}[Standard $\mathbb{R}^2$ in $\mathbb{R}^4$]
            As an example, we calculate the end and co-end Khovanov homologies for the standard $\mathbb{R}^2$ in $\mathbb{R}^4$ (i.e., the $xy$-plane of 4-space). Using the standard exhaustion with $B_i$ being the 4-ball of radius $i$ (with the exception of $B_0$ being a small 4-ball in $B_1$ not intersecting the $xy$-plane), we get the following calculations: 
            \begin{align*}
                \overrightharpoonup{Kh}^{h,q}(\mathbb{R}^2)&:=\colim_{i\in\mathbb{Z}} \left( Kh^{h,q-1}(\varnothing)\xrightarrow{\iota_\ast} Kh^{h,q}(U)\xrightarrow{\id} Kh^{h,q}(U)\xrightarrow{\id} \cdots \right)\\
                &\cong Kh^{h,q}(U) \\
                &\cong \begin{cases}
                    \mathbb{Z}/2 & \text{if } (h,q)=(0,\pm 1)\\
                    0 & \text{else}
                \end{cases}
            \end{align*}
            \begin{align*}
                \overleftharpoonup{Kh}^{h,q}(\mathbb{R}^2)&:=\varprojlim_{i\in\mathbb{Z}} \left( Kh^{h,q+1}(\varnothing)\xleftarrow{\iota^\ast} Kh^{h,q}(U)\xleftarrow{\id} Kh^{h,q}(U)\xleftarrow{\id} \cdots \right)\\
                &\cong Kh^{h,q}(U)\\
                &\cong \begin{cases}
                    \mathbb{Z}/2 & \text{if } (h,q)=(0,\pm 1)\\
                    0 & \text{else}
                \end{cases}
            \end{align*}
        \end{eg}

\section{The plane is exotic}{\label{exotic-construction}}

    In this section, we construct the surface $\Sigma$ shown in Figure \ref{20250726-1} by describing a compact exhaustion, which we will use to prove that the surface $\Sigma$ is an exotic plane and can be smoothly isotoped to a Lagrangian surface in $\mathbb{R}^4$.

    Let $B_0$ be a small 4-ball centered at the origin of $\mathbb{R}^4$, and let $B_i$ be the standard ball of radius $i$ centered at the origin. Then the family $\{B_i\}_{i\in\mathbb{Z}}$ is a compact exhaustion for the standard $\mathbb{R}^4$. Just like the setup in Section \ref{sec-end}, let $S_i$ be the boundary of $B_i$, and let $W_i:=B_i\backslash B_{i-1}$ be the cobordism between $S_{i-1}$ and $S_i$. We describe the surface $\Sigma$ by describing its intersection with each $S_i$, which we denote by $L_i$, its intersection with $B_i$, which we denote by $\Sigma_i$, and its intersection with each $W_i$, which we denote by $C_i$.

    Let $L_0$ be the empty link in $B_0$. Let $\Sigma_1=C_1$ be the ribbon surface (topologically, a disjoint union of two disks) shown on the left of Figure \ref{20250409-1-2}, and let $L_1$ be its boundary. 
    
    Now we extend the link $L_1$ to the sphere of radius 1.5 via a trivial cobordism. If we add a (2-dimensional) 0-handle and a 1-handle as shown in the right side of Figure \ref{20250409-1-2} and extend the surgered link to the sphere of radius 2, we obtain the ribbon surface $\Sigma_2$, whose boundary we denote by $L_2$. 
        
        \begin{figure}
            \centering
            \begin{picture}(12.97cm,4cm)
                \put(0cm,0cm){\includegraphics[width=3.57cm, height=3.2cm]{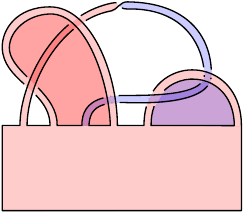}}
                \put(5cm,0cm){\includegraphics[width=7.97cm, height=4cm]{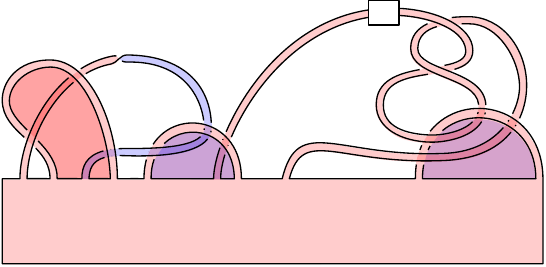}}
                                
                \put(10.4cm, 3.705cm){\scriptsize{$-2$}}
                
            \end{picture}        
            \caption{Left: The first stage $\Sigma_1$ of the exotic plane, which is a disjoint union of two disks. Right: The second stage $\Sigma_2$ of the exotic plane.}
            \label{20250409-1-2}
        \end{figure}

    Inductively, $\Sigma_i$ is constructed from $\Sigma_{i-1}$ by extending the link $L_{i-1}$ to the sphere of radius $i-1/2$, adding a 2-dimensional 0-handle and 1-handle using the same pattern, and extending the surgered link to $S_i$. Its boundary $L_i$ and the corresponding 2-dimensional cobordism $C_i$ are obtained accordingly. As an example, $C_3$ and $L_3$ are shown in Figure \ref{20250409-3}.

        \begin{figure}
            \centering
            \begin{picture}(12.36cm,4cm)
                \put(0cm,0cm){\includegraphics[width=12.36cm, height=4cm]{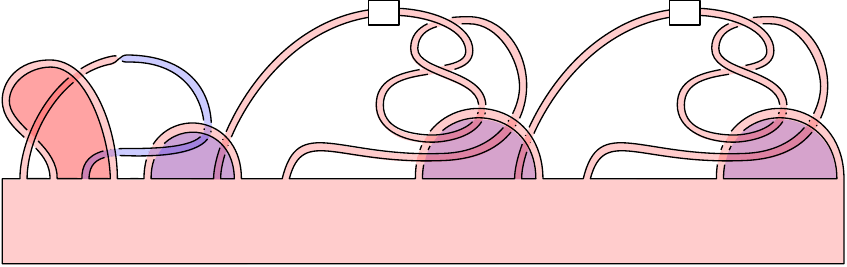}}
                                
                \put(5.4cm, 3.705cm){\scriptsize{$-2$}}
                \put(9.78cm, 3.705cm){\scriptsize{$-2$}}
                
            \end{picture}        
            \caption{The third stage $\Sigma_3$ of the exotic plane.}
            \label{20250409-3}
        \end{figure}

    In this sense, we obtain a nested sequence of compact surfaces $\Sigma_0\subset \Sigma_1\subset \Sigma_2\subset \cdots$ in $\mathbb{R}^4$. Let $\Sigma$ be the limiting surface $\bigcup_{i\in \mathbb{N}}\Sigma_i$. It is straightforward to see that the surface $\Sigma$ is homeomorphic to $\mathbb{R}^2$ and that the embedding $\Sigma\hookrightarrow\mathbb{R}^4$ is smooth and proper.

    In the rest of this paper, we will take advantage of this compact exhaustion and prove that the surface $\Sigma$ is an exotic plane and can be smoothly isotoped to a Lagrangian surface.

    \subsection{The plane is topologically standard}

        To prove that the plane $\Sigma$ is topologically standard, we apply the following theorem, which is known by experts (e.g., very briefly mentioned in Remark 5.3b in \cite{Gom25}) but, to the author's knowledge, has never been explicitly stated in the literature. Recall that for a manifold $X$ with a single end (which we denote by $\infty$), its fundamental group at infinity $\pi_1^\infty(X)$ is defined to be $\varprojlim \{\pi_1(X\backslash K) \,|\, K\subset X \text{ is compact}\}.$

        \begin{thrm}
            Let $\Sigma$ be a locally flat, properly embedded plane in $\mathbb{R}^4$. If $\pi_1(\mathbb{R}^4\backslash{\Sigma})\cong \mathbb{Z}$ and $\pi_1^\infty(\mathbb{R}^4\backslash{\Sigma})\cong \mathbb{Z}$, then $\Sigma$ is topologically isotopic to the standard plane in $\mathbb{R}^4$.
        \end{thrm}

        \begin{proof}
            We take the one-point compactification of the pair $(\mathbb{R}^4,\Sigma)$. On the 4-manifold level, we obtain the standard 4-sphere $S^4=\mathbb{R}^4\cup \{p\}$, where $p$ denotes the point at infinity. On the surface level, we obtain the topological space $\Sigma':=\Sigma\cup\{p\}$, which is homeomorphic to the 2-sphere $S^2$ since $\Sigma$ is a proper embedding.
            
            Note that $\pi_1^\infty(\mathbb{R}^4\backslash {\Sigma})\cong \mathbb{Z}$ implies that the compactified surface $\Sigma'$ is locally $1$-alg at the point $p$ (i.e., for any neighborhood $U$ of $p$ in $S^4$, there exists a neighborhood $V$ of $p$ in $U$ such that the image of $\pi_1(V\backslash \Sigma')$ in $\pi_1(U\backslash \Sigma')$ is abelian). Then the corollary after Theorem 1.3 in \cite{Ven97} (also Theorem 9.3A of \cite{FQ90}, but its proof contains a gap) ensures that $\Sigma'$ is actually locally flat. 

            Now $\Sigma'$ is a locally flat embedding of $S^2$ in $S^4$ with fundamental group of the complement being infinite cyclic. By Theorem 11.7A of \cite{FQ90}, we know that $\Sigma'$ is topologically standard. Finally, by removing the point $p$ at infinity, we conclude that $\Sigma$ is also topologically standard.
        \end{proof}

        \begin{remark}
            In the statement of this theorem, when using the notation $\mathbb{R}^4\backslash \Sigma$, we genuinely mean the setwise difference between $\mathbb{R}^4$ and $\Sigma$ (instead of $\mathbb{R}^4\backslash \mathring{\nu}(\Sigma)$, which is more conventional in the compact setting). This distinction is important because $\mathbb{R}^4\backslash\Sigma$ and $\mathbb{R}^4\backslash \mathring{\nu}(\Sigma)$ have different compact exhaustions, which may lead to different fundamental groups at infinity.
        \end{remark}
        
        With this theorem, we can now prove that the plane $\Sigma$ is topologically standard.

        \subsubsection{$\pi_1(\mathbb{R}^4\backslash {\Sigma})\cong \mathbb{Z}$}{\label{pi-1}}

            This can be proved diagrammatically. As the surface $\Sigma$ is ribbon, there is a standard way to obtain a Kirby diagram for its complement (cf. Section 6.2 of \cite{GS99}). This yields the infinite Kirby diagram shown in Figure \ref{20250408-1}.

            \begin{figure}
                \centering
                \begin{picture}(17cm,4cm)
                    \put(0cm,0cm){\includegraphics[width=17cm, height=4cm]{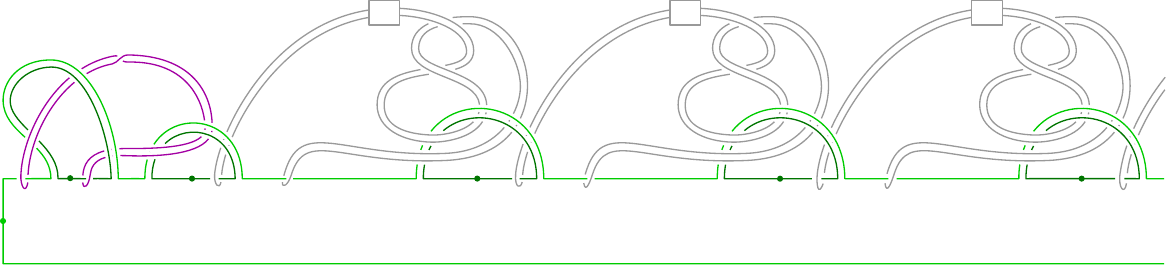}}
                                    
                    \put(5.4cm, 3.705cm){\scriptsize{\color{gray}{$-2$}}}
                    \put(9.78cm, 3.705cm){\scriptsize{\color{gray}{$-2$}}}
                    \put(14.18cm, 3.705cm){\scriptsize{\color{gray}{$-2$}}}

                    \put(4.8cm, 3.805cm){\scriptsize{\color{gray}{$0$}}}
                    \put(9.18cm, 3.805cm){\scriptsize{\color{gray}{$0$}}}
                    \put(13.58cm, 3.805cm){\scriptsize{\color{gray}{$0$}}}
                    \put(2.8cm, 3.005cm){\scriptsize{\color{purple}{$0$}}}
                    
                \end{picture}
                \caption{A Kirby diagram for the noncompact 4-manifold $\mathbb{R}^4\backslash {\Sigma}.$\vspace{-3pt}}
                \label{20250408-1}
            \end{figure}
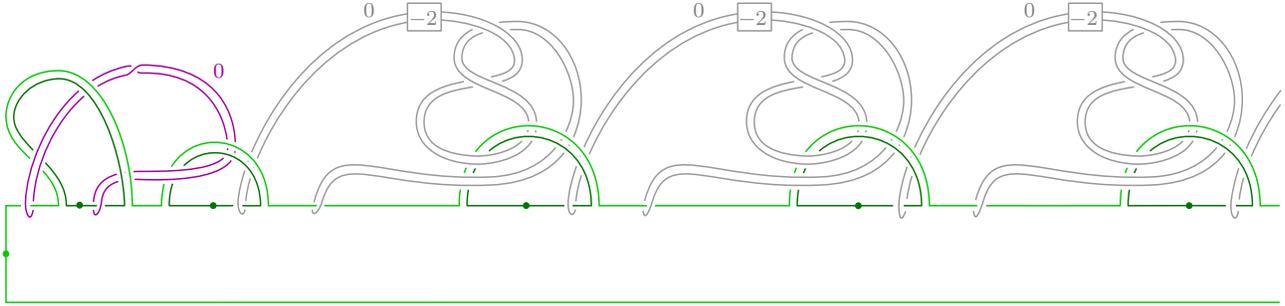

            Note that this is a noncompact handle diagram, meaning that all handles are attached to the 3-space $\mathbb{R}^3$ bounding the half 4-space $\mathbb{H}^4$ (instead of a compact 4-dimensional 0-handle). Also, we remove the boundary after all handles have been attached. 
            
            It is also worth noting that the light green dotted curve represents a ``noncompact 1-handle", as shown in the picture on the left side of Figure \ref{20250830-1}. This means we carve out a tubular neighborhood of the half-plane it bounds after pushing it into the interior of $\mathbb{H}^4$. The resulting manifold is $\mathbb{R}^4\backslash\mathbb{R}^2$. Another way of drawing $\mathbb{R}^4\backslash\mathbb{R}^2$ is by drawing infinitely many alternating compact 1-handles and 0-framed 2-handles, as illustrated in the diagram on the right side of Figure \ref{20250830-1}. The two diagrams can be identified by repeatedly sliding the leftmost compact 1-handle over the 1-handle to the right and canceling excessive 1/2-handle pairs. In this paper, we'll use the first notation primarily as a shorthand for the second.

            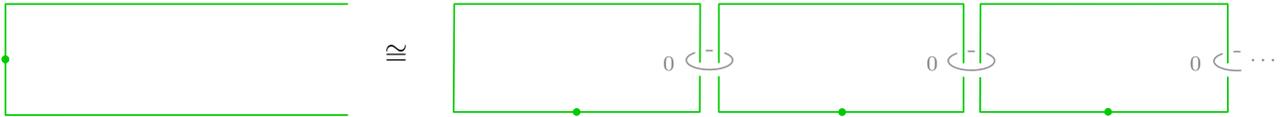
\begin{figure}
                \centering
                \begin{picture}(16.5cm,1.5cm)      

                    \put(0cm,-0.1cm){\includegraphics[height=1.5cm]{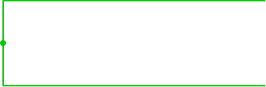}}
                    \put(6cm,-0.1cm){\includegraphics[height=1.5cm]{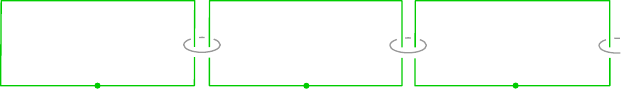}}
                    \put(5.1cm, 0.6cm){$\cong$}

                    \put(8.8cm, 0.5cm){\scriptsize{\color{gray}{$0$}}}
                    \put(12.3cm, 0.5cm){\scriptsize{\color{gray}{$0$}}}
                    \put(15.8cm, 0.5cm){\scriptsize{\color{gray}{$0$}}}
                    \put(16.6cm, 0.57cm){\scriptsize{\color{gray}{$\cdots$}}}

                \end{picture}        
                \caption{Two ways of drawing a noncompact 1-handle.\vspace{-3pt}}
                \label{20250830-1}
            \end{figure}

            Now we return to the Kirby diagram of $\mathbb{R}^4\backslash \Sigma$ shown in Figure \ref{20250408-1}. As the fundamental group of a 4-manifold depends only on the homotopy classes of the attaching circles of its 2-handles, we are free to modify our 4-manifold (without changing the fundamental group) by undoing the clasps and twist boxes on the gray 2-handles, obtaining the Kirby diagram in Figure \ref{20250408-2}. After isotopy, we obtain the Kirby diagram in Figure \ref{20250408-3}. 

            \begin{figure}
                \centering
                \begin{picture}(17cm,3.85cm)
                    \put(0cm,0cm){\includegraphics[width=17cm, height=3.85cm]{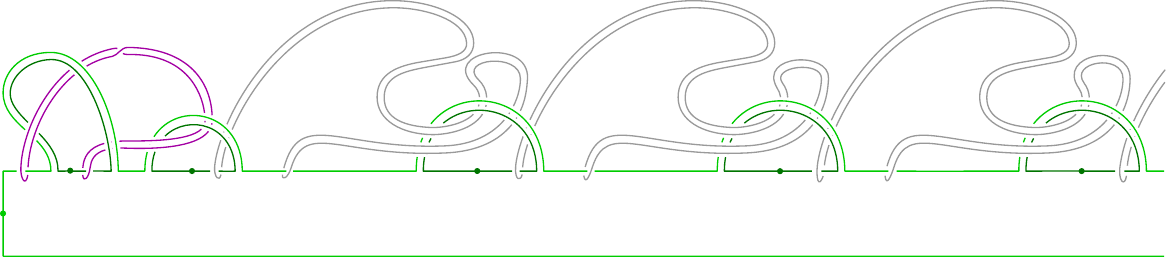}}

                    \put(4.8cm, 3.805cm){\scriptsize{\color{gray}{$0$}}}
                    \put(9.18cm, 3.805cm){\scriptsize{\color{gray}{$0$}}}
                    \put(13.58cm, 3.805cm){\scriptsize{\color{gray}{$0$}}}
                    \put(2.8cm, 3.005cm){\scriptsize{\color{purple}{$0$}}}
                    
                \end{picture}
                \caption{A 4-manifold that is homotopy equivalent to $\mathbb{R}^4\backslash\Sigma$.\vspace{-9pt}}
                \label{20250408-2}
            \end{figure}
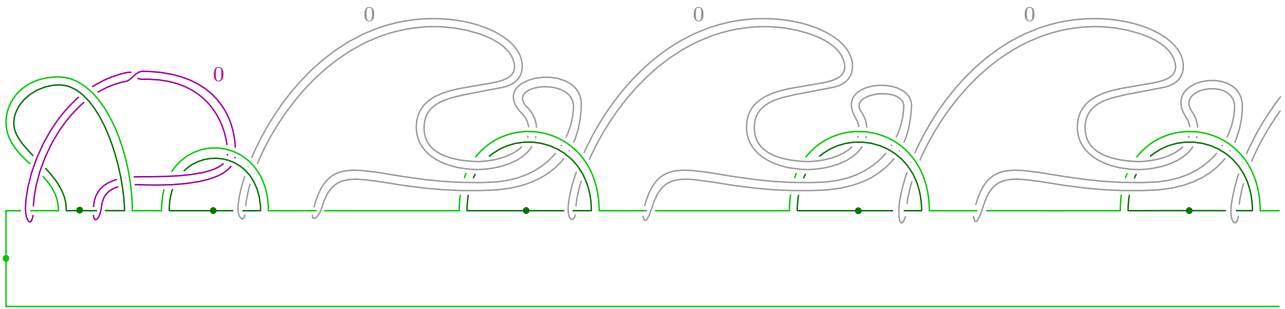

            \begin{figure}
                \centering
                \begin{picture}(17cm,3.2cm)
                    \put(0cm,0cm){\includegraphics[width=17cm, height=3.2cm]{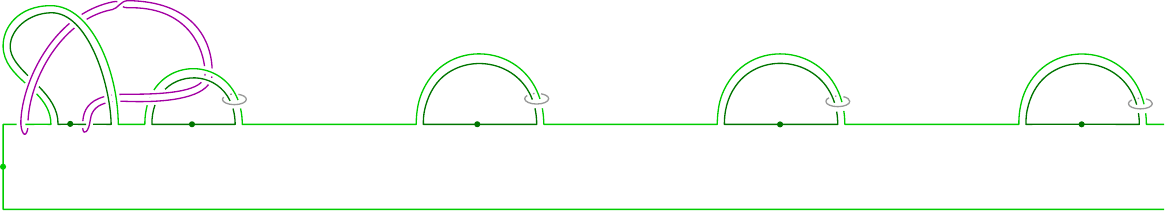}}

                    \put(3.7cm, 1.805cm){\scriptsize{\color{gray}{$0$}}}
                    \put(8.08cm, 1.805cm){\scriptsize{\color{gray}{$0$}}}
                    \put(12.48cm, 1.805cm){\scriptsize{\color{gray}{$0$}}}
                    \put(16.88cm, 1.805cm){\scriptsize{\color{gray}{$0$}}}
                    \put(2.8cm, 3.005cm){\scriptsize{\color{purple}{$0$}}}
                    
                \end{picture}
                \caption{}
                \label{20250408-3}
            \end{figure}

            Now each gray 2-handle cancels a dark green 1-handle (after sliding the bright green 1-handle over the dark green 1-handles), which yields the leftmost Kirby diagram in Figure \ref{20250408-5-6-7}. An isotopy gives the middle Kirby diagram. After a final cancellation of the purple/dark-green 1/2-handle pair, we obtain the rightmost Kirby diagram of Figure \ref{20250408-5-6-7}. This represents the complement of the standard $\mathbb{R}^2$ in $\mathbb{R}^4$, which has fundamental group $\mathbb{Z}$.

            \begin{figure}[H]
                \centering
                \begin{picture}(16cm,3.2cm)      

                    \put(0cm,0cm){\includegraphics[height=3.2cm]{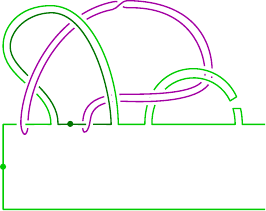}}
                    \put(6cm,0cm){\includegraphics[height=3.1cm]{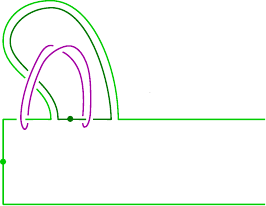}}
                    \put(12cm,0cm){\includegraphics[height=1.7cm]{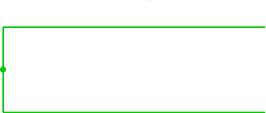}}

                    \put(4.8cm,1.2cm){$\cong$}
                    \put(10.8cm,1.2cm){$\cong$}

                    \put(2.9cm, 3.005cm){\scriptsize{\color{purple}{$0$}}}
                    \put(6.8cm, 2.515cm){\scriptsize{\color{purple}{$0$}}}

                \end{picture}        
                \caption{}
                \label{20250408-5-6-7}
            \end{figure}

        \subsubsection{$\pi_1^\infty(\mathbb{R}^4\backslash \Sigma)\cong \mathbb{Z}$}

            Again we recall that for a manifold $X$ with a single end (denoted by $\infty$), its fundamental group at infinity $\pi_1^\infty(X)$ is defined to be $\varprojlim \{\pi_1(X\backslash K) \,|\, K\subset X \text{ is compact}\}.$

            Next we describe a compact exhaustion for the open manifold $\mathbb{R}^4\backslash \Sigma.$ We would like to emphasize again that the notation $\mathbb{R}^4\backslash \Sigma$ here stands for setwise difference instead of $\mathbb{R}^4\backslash \mathring{\nu}(\Sigma)$. Let $\{\epsilon_i\}$ be a small enough decreasing sequence that converges to $0$, and let $\nu_{\epsilon_i}(\Sigma_i)$ be the diameter $\epsilon_i$ (trivial) tubular neighborhood of $\Sigma_i$ in $B_i$. Now $B_0\backslash \mathring{\nu}_{\epsilon_0}(\Sigma_0) \subset B_1\backslash \mathring{\nu}_{\epsilon_1}(\Sigma_1) \subset B_2\backslash \mathring{\nu}_{\epsilon_2}(\Sigma_2) \subset \cdots$ gives an actual compact exhaustion for the open manifold $\mathbb{R}^4\backslash\Sigma$. A schematic can be seen in Figure \ref{compact-exhaustion}.

            \begin{figure}[H]
                \centering
                \begin{picture}(7cm,4.0cm)
                    \put(0cm,0.0cm){\includegraphics[height=4.0cm]{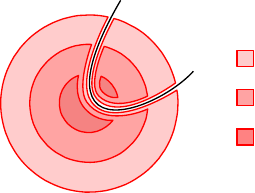}}

                    \put(4.15cm,2.4cm){$\Sigma$}
                    \put(5.5cm,1.03cm){\textcolor{red}{$B_1\backslash\mathring{\nu}_{\epsilon_1}(\Sigma_1)$}}
                    \put(5.5cm,1.83cm){\textcolor{red}{$B_2\backslash\mathring{\nu}_{\epsilon_2}(\Sigma_2)$}}
                    \put(5.5cm,2.63cm){\textcolor{red}{$B_3\backslash\mathring{\nu}_{\epsilon_3}(\Sigma_3)$}}
                \end{picture}        
                \caption{A schematic for a compact exhaustion of the manifold $\mathbb{R}^4\backslash \Sigma$.}
                \label{compact-exhaustion}
            \end{figure}
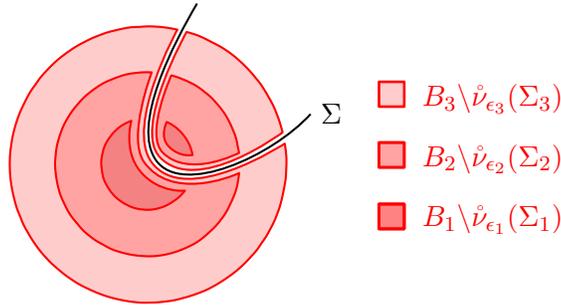 

            \noindent This compact exhaustion induces an inverse system $$\pi_1((\mathbb{R}^4\backslash\Sigma) - (B_0\backslash \mathring{\nu}_{\epsilon_0}(\Sigma_0))) \xleftarrow{i_\ast} \pi_1((\mathbb{R}^4\backslash\Sigma) - (B_1\backslash \mathring{\nu}_{\epsilon_1}(\Sigma_1))) \xleftarrow{i_\ast} \pi_1((\mathbb{R}^4\backslash\Sigma) - (B_2\backslash \mathring{\nu}_{\epsilon_2}(\Sigma_2))) \xleftarrow{i_\ast} \cdots$$

            \noindent By the uniqueness of inverse limits, we know that the inverse limit of this system is isomorphic to the fundamental group at infinity $\pi_1^\infty(\mathbb{R}^4\backslash\Sigma)$.

            We claim that for any positive integer $i$, the fundamental group $\pi_1((\mathbb{R}^4\backslash\Sigma) - (B_i\backslash \mathring{\nu}_{\epsilon_i}(\Sigma_i)))$ is isomorphic to $\mathbb{Z}$, generated by the meridian of the surface $\Sigma\backslash\Sigma_i$. If this claim is indeed true, the inverse system above is just the system $\mathbb{Z}\xleftarrow{\cong}\mathbb{Z} \xleftarrow{\cong}\cdots$, whose inverse limit is $\mathbb{Z}$. This shows that $\pi_1^\infty(\mathbb{R}^4\backslash \Sigma)\cong \mathbb{Z}$.

            It remains to show that $\pi_1((\mathbb{R}^4\backslash\Sigma) - (B_i\backslash \mathring{\nu}_{\epsilon_i}(\Sigma_i))) \cong \mathbb{Z}$ for all $i>0.$ As a first step, we need a Kirby diagram description for this 4-manifold. Note that the boundary of $(\mathbb{R}^4\backslash\Sigma) - (B_i\backslash \mathring{\nu}_{\epsilon_i}(\Sigma_i))$ is diffeomorphic to the boundary of $B_i\backslash \mathring{\nu}_{\epsilon_i}(\Sigma_i)$, which we denote by $\partial_-$. Then it is straightforward to see that the 4-manifold $(\mathbb{R}^4\backslash\Sigma) - (B_i\backslash \mathring{\nu}_{\epsilon_i}(\Sigma_i))$ can be built from (a thickening of) the boundary $\partial_-$ by stacking the infinite sequence of cobordisms $\bigcup_{k>i}W_k\backslash \mathring{\nu}_{\epsilon_k}(C_k)$.\footnote{Technically, we should be stacking the cobordisms $(B_k\backslash \mathring{\nu}_{\epsilon_k}(\Sigma_k))-(\mathring{B}_{k-1}\backslash \mathring{\nu}_{\epsilon_{k-1}}(\Sigma_{k-1}))$, which differ from $W_k\backslash \mathring{\nu}_{\epsilon_k}(C_k)$ by a thin piece parallel to $\Sigma \cap B_k$. However, this difference becomes irrelevant (up to diffeomorphism) when these cobordisms are stacked and glued to the thickening of $\partial_-$.}

            Thus, to draw its Kirby diagram, we just modify the Kirby diagram of $\mathbb{R}^4\backslash\Sigma$ (Figure \ref{20250408-1}) by replacing the first $i$ stages with its boundary surgery diagram (by replacing each 1-handle and 0-framed 2-handle curve with $\langle 0\rangle$-framed surgery curves). Then the surgery diagram part illustrates the frontier $\partial_-$, while the handle diagram part illustrates the infinite cobordism on top of $\partial_-$. A typical example $(\mathbb{R}^4\backslash\Sigma) - (B_3\backslash \mathring{\nu}(\Sigma_3))$ is shown in Figure \ref{20250503-1}.

            \begin{figure}[H]
                \centering
                \begin{picture}(17cm,4cm)
                    \put(0cm,0cm){\includegraphics[width=17cm, height=4cm]{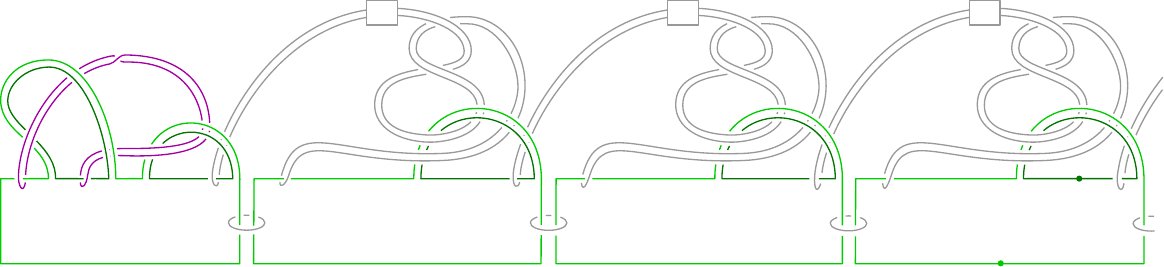}}
                                    
                    \put(5.4cm, 3.705cm){\scriptsize{\color{gray}{$-2$}}}
                    \put(9.78cm, 3.705cm){\scriptsize{\color{gray}{$-2$}}}
                    \put(14.18cm, 3.705cm){\scriptsize{\color{gray}{$-2$}}}

                    \put(4.4cm, 3.805cm){\scriptsize{\color{gray}{$\langle 0\rangle$}}}
                    \put(8.78cm, 3.805cm){\scriptsize{\color{gray}{$\langle 0\rangle$}}}
                    \put(13.18cm, 3.805cm){\scriptsize{\color{gray}{$0$}}}
                    
                    \put(2.8cm, 3.005cm){\scriptsize{\color{purple}{$\langle 0\rangle$}}}

                    \put(0.05cm, 0.605cm){\scriptsize{\color{green}{$\langle 0\rangle$}}}
                    \put(5.55cm, 0.155cm){\scriptsize{\color{green}{$\langle 0\rangle$}}}
                    \put(10.05cm, 0.155cm){\scriptsize{\color{green}{$\langle 0\rangle$}}}

                    \put(3.95cm, 0.605cm){\scriptsize{\color{gray}{$\langle 0\rangle$}}}
                    \put(8.35cm, 0.605cm){\scriptsize{\color{gray}{$\langle 0\rangle$}}}
                    \put(12.75cm, 0.605cm){\scriptsize{\color{gray}{$0$}}}
                    \put(16.35cm, 0.605cm){\scriptsize{\color{gray}{$0$}}}

                    \put(0.75cm, 1.045cm){\scriptsize{\color{forest}{$\langle 0\rangle$}}}
                    \put(2.55cm, 1.045cm){\scriptsize{\color{forest}{$\langle 0\rangle$}}}
                    \put(6.75cm, 1.045cm){\scriptsize{\color{forest}{$\langle 0\rangle$}}}
                    \put(11.15cm, 1.045cm){\scriptsize{\color{forest}{$\langle 0\rangle$}}}

                \end{picture}
                \caption{A relative Kirby diagram for the manifold $(\mathbb{R}^4\backslash\Sigma) - (B_3\backslash \mathring{\nu}(\Sigma_3))$.}
                \label{20250503-1}
            \end{figure}

            Now that we have a Kirby diagram for $(\mathbb{R}^4\backslash\Sigma) - (B_i\backslash \mathring{\nu}_{\epsilon_i}(\Sigma_i))$, we can start computing its fundamental group. As in subsection \ref{pi-1}, since the fundamental group depends only on the homotopy classes of the attaching circles of the 2-handles, we can undo the clasps and twist boxes on the gray 2-handles. Like before, we obtain the Kirby diagram in Figure \ref{20250503-2}. A 2-handle slide as indicated by the dotted arrow on the righthand side yields the diagram in Figure \ref{20250503-3}.

            \begin{figure}[H]
                \centering
                \begin{picture}(16.8cm,4cm)
                    \put(0cm,0cm){\includegraphics[width=16.8cm, height=4cm]{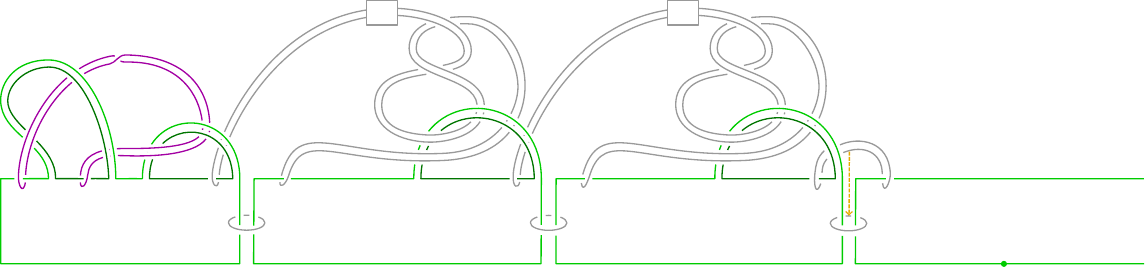}}
                                    
                    \put(5.4cm, 3.705cm){\scriptsize{\color{gray}{$-2$}}}
                    \put(9.81cm, 3.705cm){\scriptsize{\color{gray}{$-2$}}}

                    \put(4.4cm, 3.805cm){\scriptsize{\color{gray}{$\langle 0\rangle$}}}
                    \put(8.78cm, 3.805cm){\scriptsize{\color{gray}{$\langle 0\rangle$}}}
                    \put(13.08cm, 1.805cm){\scriptsize{\color{gray}{$0$}}}
                    
                    \put(2.8cm, 3.005cm){\scriptsize{\color{purple}{$\langle 0\rangle$}}}

                    \put(0.05cm, 0.605cm){\scriptsize{\color{green}{$\langle 0\rangle$}}}
                    \put(5.55cm, 0.155cm){\scriptsize{\color{green}{$\langle 0\rangle$}}}
                    \put(10.05cm, 0.155cm){\scriptsize{\color{green}{$\langle 0\rangle$}}}

                    \put(3.95cm, 0.605cm){\scriptsize{\color{gray}{$\langle 0\rangle$}}}
                    \put(8.40cm, 0.605cm){\scriptsize{\color{gray}{$\langle 0\rangle$}}}
                    \put(12.85cm, 0.605cm){\scriptsize{\color{gray}{$0$}}}

                    \put(0.75cm, 1.045cm){\scriptsize{\color{forest}{$\langle 0\rangle$}}}
                    \put(2.55cm, 1.045cm){\scriptsize{\color{forest}{$\langle 0\rangle$}}}
                    \put(6.75cm, 1.045cm){\scriptsize{\color{forest}{$\langle 0\rangle$}}}
                    \put(11.15cm, 1.045cm){\scriptsize{\color{forest}{$\langle 0\rangle$}}}

                \end{picture}
                \caption{}
                \label{20250503-2}
            \end{figure}

            \begin{figure}[H]
                \centering
                \begin{picture}(16.8cm,4cm)
                    \put(0cm,0cm){\includegraphics[width=16.8cm, height=4cm]{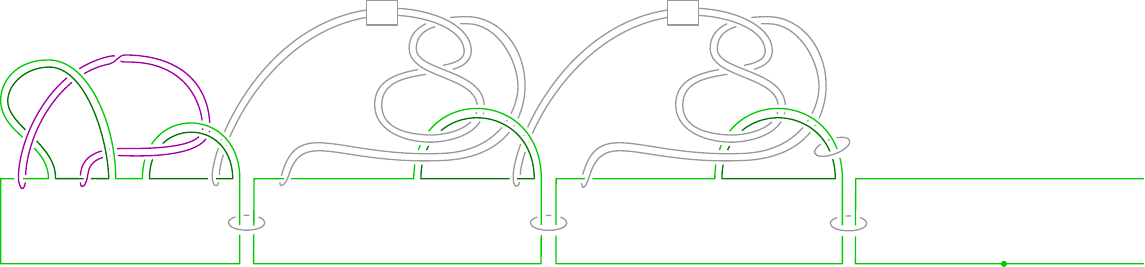}}
                                    
                    \put(5.4cm, 3.705cm){\scriptsize{\color{gray}{$-2$}}}
                    \put(9.81cm, 3.705cm){\scriptsize{\color{gray}{$-2$}}}

                    \put(4.4cm, 3.805cm){\scriptsize{\color{gray}{$\langle 0\rangle$}}}
                    \put(8.78cm, 3.805cm){\scriptsize{\color{gray}{$\langle 0\rangle$}}}
                    \put(12.58cm, 1.805cm){\scriptsize{\color{gray}{$0$}}}
                    
                    \put(2.8cm, 3.005cm){\scriptsize{\color{purple}{$\langle 0\rangle$}}}

                    \put(0.05cm, 0.605cm){\scriptsize{\color{green}{$\langle 0\rangle$}}}
                    \put(5.55cm, 0.155cm){\scriptsize{\color{green}{$\langle 0\rangle$}}}
                    \put(10.05cm, 0.155cm){\scriptsize{\color{green}{$\langle 0\rangle$}}}

                    \put(3.95cm, 0.605cm){\scriptsize{\color{gray}{$\langle 0\rangle$}}}
                    \put(8.40cm, 0.605cm){\scriptsize{\color{gray}{$\langle 0\rangle$}}}
                    \put(12.85cm, 0.605cm){\scriptsize{\color{gray}{$0$}}}

                    \put(0.75cm, 1.045cm){\scriptsize{\color{forest}{$\langle 0\rangle$}}}
                    \put(2.55cm, 1.045cm){\scriptsize{\color{forest}{$\langle 0\rangle$}}}
                    \put(6.75cm, 1.045cm){\scriptsize{\color{forest}{$\langle 0\rangle$}}}
                    \put(11.15cm, 1.045cm){\scriptsize{\color{forest}{$\langle 0\rangle$}}}

                \end{picture}
                \caption{}
                \label{20250503-3}
            \end{figure}

            Next we slide the bright green surgery curves over the corresponding dark green curves (which are slides within the 3-dimensional surgery diagrams of $\partial_-$). This gives the diagram in Figure \ref{20250503-4}.

            \begin{figure}[H]
                \centering
                \begin{picture}(16.8cm,4cm)
                    \put(0cm,0cm){\includegraphics[width=16.8cm, height=4cm]{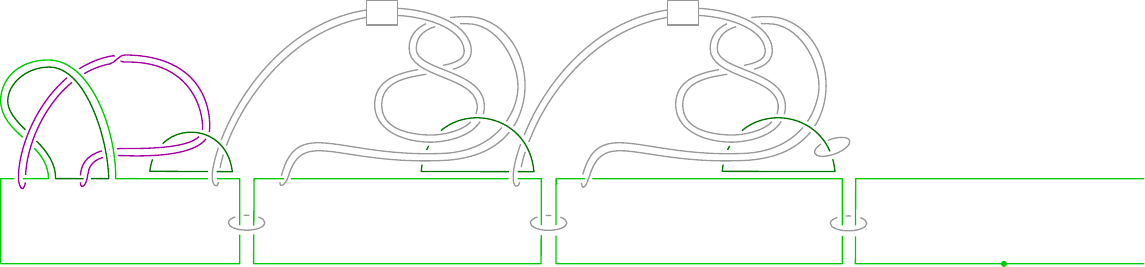}}
                                    
                    \put(5.4cm, 3.705cm){\scriptsize{\color{gray}{$-2$}}}
                    \put(9.81cm, 3.705cm){\scriptsize{\color{gray}{$-2$}}}

                    \put(4.4cm, 3.805cm){\scriptsize{\color{gray}{$\langle 0\rangle$}}}
                    \put(8.78cm, 3.805cm){\scriptsize{\color{gray}{$\langle 0\rangle$}}}
                    \put(12.58cm, 1.805cm){\scriptsize{\color{gray}{$0$}}}
                    
                    \put(2.8cm, 3.005cm){\scriptsize{\color{purple}{$\langle 0\rangle$}}}

                    \put(0.05cm, 0.605cm){\scriptsize{\color{green}{$\langle 0\rangle$}}}
                    \put(5.55cm, 0.155cm){\scriptsize{\color{green}{$\langle 0\rangle$}}}
                    \put(10.05cm, 0.155cm){\scriptsize{\color{green}{$\langle 0\rangle$}}}

                    \put(3.95cm, 0.605cm){\scriptsize{\color{gray}{$\langle 0\rangle$}}}
                    \put(8.40cm, 0.605cm){\scriptsize{\color{gray}{$\langle 0\rangle$}}}
                    \put(12.85cm, 0.605cm){\scriptsize{\color{gray}{$0$}}}

                    \put(0.75cm, 1.045cm){\scriptsize{\color{forest}{$\langle 0\rangle$}}}
                    \put(2.55cm, 1.045cm){\scriptsize{\color{forest}{$\langle 0\rangle$}}}
                    \put(6.75cm, 1.045cm){\scriptsize{\color{forest}{$\langle 0\rangle$}}}
                    \put(11.15cm, 1.045cm){\scriptsize{\color{forest}{$\langle 0\rangle$}}}

                \end{picture}
                \caption{}
                \label{20250503-4}
            \end{figure}

            Next we observe the Kirby diagram shown in Figure \ref{20250503-5}, where we remove the $\langle 0\rangle$-framed surgery curve together with its meridional 0-framed 2-handle. This creates a new 4-manifold $X_i$ whose fundamental group surjects onto $\pi_1((\mathbb{R}^4\backslash\Sigma) - (B_i\backslash \mathring{\nu}_{\epsilon_i}(\Sigma_i)))$. We will show that $\pi_1(X_i)\cong \mathbb{Z}$.

            \begin{figure}[H]
                \centering
                \begin{picture}(16.8cm,4cm)
                    \put(0cm,0cm){\includegraphics[width=16.8cm, height=4cm]{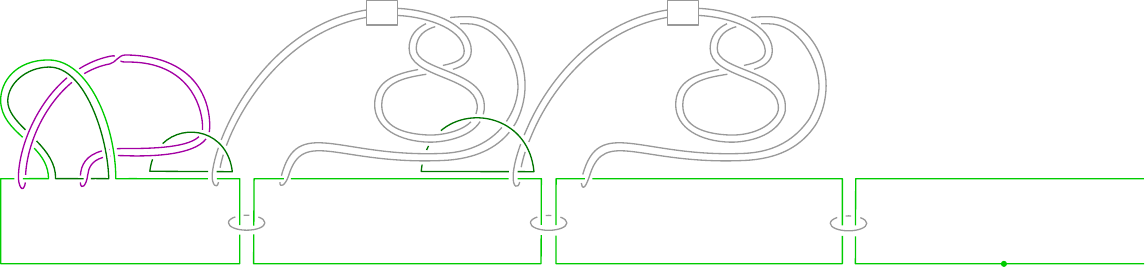}}
                                    
                    \put(5.4cm, 3.705cm){\scriptsize{\color{gray}{$-2$}}}
                    \put(9.81cm, 3.705cm){\scriptsize{\color{gray}{$-2$}}}

                    \put(4.4cm, 3.805cm){\scriptsize{\color{gray}{$\langle 0\rangle$}}}
                    \put(8.78cm, 3.805cm){\scriptsize{\color{gray}{$\langle 0\rangle$}}}
                    
                    \put(2.8cm, 3.005cm){\scriptsize{\color{purple}{$\langle 0\rangle$}}}

                    \put(0.05cm, 0.605cm){\scriptsize{\color{green}{$\langle 0\rangle$}}}
                    \put(5.55cm, 0.155cm){\scriptsize{\color{green}{$\langle 0\rangle$}}}
                    \put(10.05cm, 0.155cm){\scriptsize{\color{green}{$\langle 0\rangle$}}}

                    \put(3.95cm, 0.605cm){\scriptsize{\color{gray}{$\langle 0\rangle$}}}
                    \put(8.40cm, 0.605cm){\scriptsize{\color{gray}{$\langle 0\rangle$}}}
                    \put(12.85cm, 0.605cm){\scriptsize{\color{gray}{$0$}}}

                    \put(0.75cm, 1.045cm){\scriptsize{\color{forest}{$\langle 0\rangle$}}}
                    \put(2.55cm, 1.045cm){\scriptsize{\color{forest}{$\langle 0\rangle$}}}
                    \put(6.75cm, 1.045cm){\scriptsize{\color{forest}{$\langle 0\rangle$}}}

                \end{picture}
                \caption{The modified manifold $X_i$ whose fundamental group surjects onto $\pi_1((\mathbb{R}^4\backslash\Sigma) - (B_i\backslash \mathring{\nu}_{\epsilon_i}(\Sigma_i)))$.}
                \label{20250503-5}
            \end{figure}

            To justify the surjection, we consider a local model, illustrated in Figure \ref{local-model}. The dark green surgery curve creates a normal generator for the fundamental group as well as a possibly nontrivial relator. The 4-dimensional 2-handle on the right cancels the normal generator. This means that $\pi_1((\mathbb{R}^4\backslash\Sigma) - (B_i\backslash \mathring{\nu}_{\epsilon_i}(\Sigma_i)))$ is just $\pi_1(X_i)$ with one more possibly nontrivial relator, resulting in the surjective quotient map. 

            \begin{figure}[H]
                \centering
                \begin{picture}(3.35cm,3.0cm)
                    \put(0cm,0.0cm){\includegraphics[height=3.0cm]{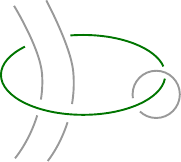}}

                    \put(0.7cm,2.1cm){\textcolor{gray}{$\cdots$}}
                    \put(3.3cm,0.6cm){\textcolor{gray}{$0$}}
                    \put(-0.1cm,0.75cm){\textcolor{forest}{$\langle 0\rangle$}}
                \end{picture}        
                \caption{A local model that illustrates the modification to $X_i$.\vspace{-5pt}}
                \label{local-model}
            \end{figure}
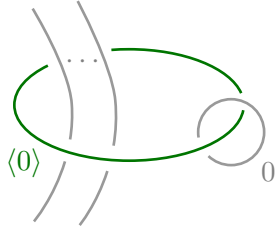 

            Next we further modify the manifold $X_i$ by removing the rightmost gray 0-framed 2-handle and the noncompact 1-handle curve, as shown in Figure \ref{20250909-1}. This modification yields a new 4-manifold $X'_i$ with the same fundamental group as $X_i$. Indeed, the fundamental group generator corresponding to the noncompact 1-handle (represented by its meridian) is identified with the generator corresponding to the bright green surgery curve on its left (due to the gray 2-handle).

            \begin{figure}[H]
                \centering
                \begin{picture}(12.35cm,4cm)
                    \put(0cm,0cm){\includegraphics[width=12.35cm, height=4cm]{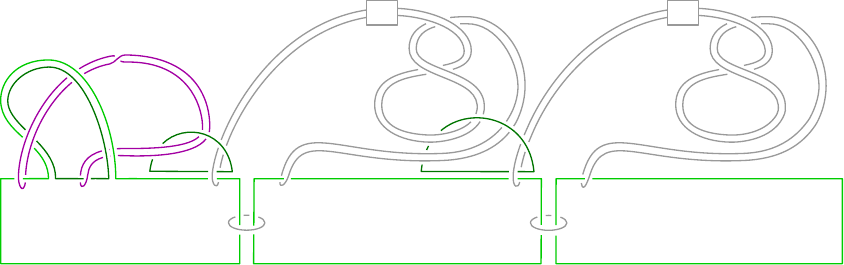}}
                                    
                    \put(5.4cm, 3.705cm){\scriptsize{\color{gray}{$-2$}}}
                    \put(9.81cm, 3.705cm){\scriptsize{\color{gray}{$-2$}}}

                    \put(4.4cm, 3.805cm){\scriptsize{\color{gray}{$\langle 0\rangle$}}}
                    \put(8.78cm, 3.805cm){\scriptsize{\color{gray}{$\langle 0\rangle$}}}
                    
                    \put(2.8cm, 3.005cm){\scriptsize{\color{purple}{$\langle 0\rangle$}}}

                    \put(0.05cm, 0.605cm){\scriptsize{\color{green}{$\langle 0\rangle$}}}
                    \put(5.55cm, 0.155cm){\scriptsize{\color{green}{$\langle 0\rangle$}}}
                    \put(10.05cm, 0.155cm){\scriptsize{\color{green}{$\langle 0\rangle$}}}

                    \put(3.95cm, 0.605cm){\scriptsize{\color{gray}{$\langle 0\rangle$}}}
                    \put(8.40cm, 0.605cm){\scriptsize{\color{gray}{$\langle 0\rangle$}}}

                    \put(0.75cm, 1.045cm){\scriptsize{\color{forest}{$\langle 0\rangle$}}}
                    \put(2.55cm, 1.045cm){\scriptsize{\color{forest}{$\langle 0\rangle$}}}
                    \put(6.75cm, 1.045cm){\scriptsize{\color{forest}{$\langle 0\rangle$}}}

                \end{picture}
                \caption{The modified (compact) manifold $X'_i$ whose fundamental group is isomorphic to $\pi_1(X_i)$. \vspace{-7pt}}
                \label{20250909-1}
            \end{figure}
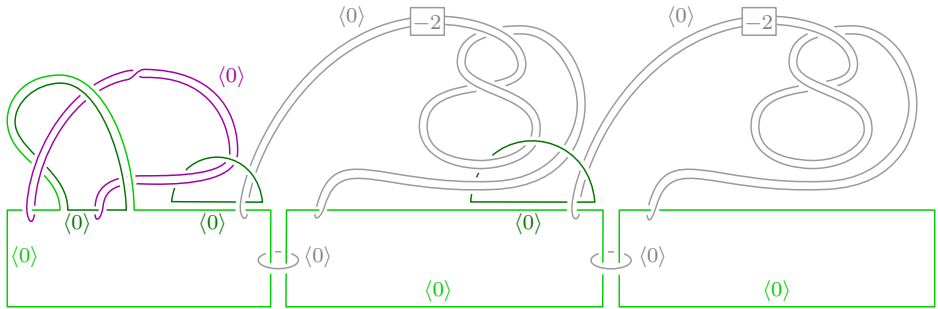

            If we consecutively slide the leftmost bright green curve over the next one on the right, we obtain canceling pairs of bright green/gray curves. Canceling them via slam dunks results in Figure \ref{20250806-1}.

            \begin{figure}[H]
                \centering
                \begin{picture}(12.35cm,3.9cm)
                    \put(0cm,0cm){\includegraphics[width=12.35cm, height=4cm]{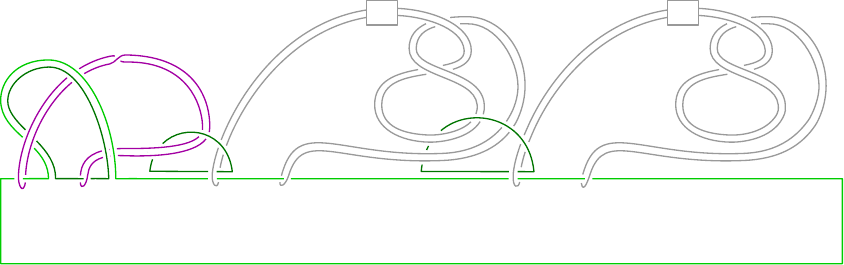}}
                                    
                    \put(5.4cm, 3.705cm){\scriptsize{\color{gray}{$-2$}}}
                    \put(9.81cm, 3.705cm){\scriptsize{\color{gray}{$-2$}}}

                    \put(4.4cm, 3.805cm){\scriptsize{\color{gray}{$\langle 0\rangle$}}}
                    \put(8.78cm, 3.805cm){\scriptsize{\color{gray}{$\langle 0\rangle$}}}
                    
                    \put(2.8cm, 3.005cm){\scriptsize{\color{purple}{$\langle 0\rangle$}}}

                    \put(0.05cm, 0.605cm){\scriptsize{\color{green}{$\langle 0\rangle$}}}

                    \put(0.75cm, 1.045cm){\scriptsize{\color{forest}{$\langle 0\rangle$}}}
                    \put(2.55cm, 1.045cm){\scriptsize{\color{forest}{$\langle 0\rangle$}}}
                    \put(6.75cm, 1.045cm){\scriptsize{\color{forest}{$\langle 0\rangle$}}}

                \end{picture}
                \caption{\vspace{-7pt}}
                \label{20250806-1}
            \end{figure}

            Now we can inductively isotope the rightmost gray surgery curve (as shown in Figure \ref{20250503-6}) and cancel the rightmost gray/dark-green surgery curves via slam-dunk (as shown in the picture on the left side of Figure \ref{20250503-7}) until we reach the diagram described in the middle of Figure \ref{20250503-7}. 

            Finally we can cancel the purple and the dark green surgery curves (after a handle slide) as before. This gives the diagram on the right side of Figure \ref{20250503-7}, which represents (a thickening of) $S^1\times S^2$. Its fundamental group is infinite cyclic, generated by the meridian of the surgery curve. 

            \begin{figure}
                \centering
                \begin{picture}(12.35cm,4cm)
                    \put(0cm,0cm){\includegraphics[width=12.35cm, height=4cm]{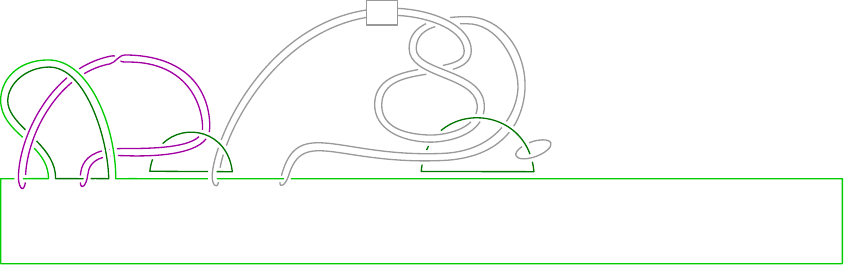}}
                                    
                    \put(5.4cm, 3.705cm){\scriptsize{\color{gray}{$-2$}}}

                    \put(4.4cm, 3.805cm){\scriptsize{\color{gray}{$\langle 0\rangle$}}}
                    \put(8.18cm, 1.805cm){\scriptsize{\color{gray}{$\langle 0\rangle$}}}
                    
                    \put(2.8cm, 3.005cm){\scriptsize{\color{purple}{$\langle 0\rangle$}}}

                    \put(0.05cm, 0.605cm){\scriptsize{\color{green}{$\langle 0\rangle$}}}

                    \put(0.75cm, 1.045cm){\scriptsize{\color{forest}{$\langle 0\rangle$}}}
                    \put(2.55cm, 1.045cm){\scriptsize{\color{forest}{$\langle 0\rangle$}}}
                    \put(6.75cm, 1.045cm){\scriptsize{\color{forest}{$\langle 0\rangle$}}}

                \end{picture}
                \caption{}
                \label{20250503-6}
            \end{figure}

            \begin{figure}
                \centering
                \begin{picture}(13.05cm,3.2cm)
                    \put(0cm,0cm){\includegraphics[width=3.65cm, height=3.2cm]{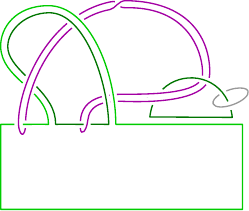}}

                    \put(3.7cm, 1.805cm){\scriptsize{\color{gray}{$\langle 0\rangle$}}}
                    
                    \put(2.8cm, 3.005cm){\scriptsize{\color{purple}{$\langle 0\rangle$}}}

                    \put(0.05cm, 0.605cm){\scriptsize{\color{green}{$\langle 0\rangle$}}}

                    \put(0.75cm, 1.045cm){\scriptsize{\color{forest}{$\langle 0\rangle$}}}
                    \put(2.55cm, 1.045cm){\scriptsize{\color{forest}{$\langle 0\rangle$}}}

                    \put(5.0cm,0cm){\includegraphics[width=3.3cm, height=3.2cm]{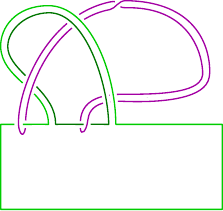}}
                    
                    \put(7.8cm, 3.005cm){\scriptsize{\color{purple}{$\langle 0\rangle$}}}

                    \put(5.05cm, 0.605cm){\scriptsize{\color{green}{$\langle 0\rangle$}}}

                    \put(5.75cm, 1.045cm){\scriptsize{\color{forest}{$\langle 0\rangle$}}}

                    \put(4.1cm, 0.5cm){$\cong$}
                    \put(8.8cm, 0.5cm){$\cong$}

                    \put(9.7cm,0cm){\includegraphics[width=3.3cm, height=1.7cm]{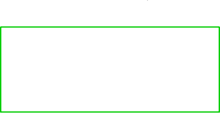}}

                    \put(9.75cm, 0.605cm){\scriptsize{\color{green}{$\langle 0\rangle$}}}
                    
                \end{picture}
                \caption{\vspace{-10pt}}
                \label{20250503-7}
            \end{figure}

            This shows that the fundamental group of the 4-manifold $X'_i$ (hence also $X_i$) is isomorphic to $\mathbb{Z}$. Since we know it surjects onto the fundamental group of $(\mathbb{R}^4\backslash\Sigma) - (B_i\backslash \mathring{\nu}_{\epsilon_i}(\Sigma_i))$, we can conclude that $\pi_1((\mathbb{R}^4\backslash\Sigma) - (B_i\backslash \mathring{\nu}_{\epsilon_i}(\Sigma_i)))$ is a cyclic group. However, it is straightforward to see that the first homology group $H_1((\mathbb{R}^4\backslash\Sigma) - (B_i\backslash \mathring{\nu}_{\epsilon_i}(\Sigma_i)))$ is isomorphic to $\mathbb{Z}$. This forces the fundamental group $\pi_1((\mathbb{R}^4\backslash\Sigma) - (B_i\backslash \mathring{\nu}_{\epsilon_i}(\Sigma_i)))$ to be infinite cyclic.
            
            This completes the proof that $\Sigma$ is topologically standard. \hfill$\square$

    \subsection{The plane is smoothly non-standard}

        In this subsection, we use end Khovanov homology to prove that the surface $\Sigma$ is not diffeomorphic to the standard $\mathbb{R}^2$ in $\mathbb{R}^4$. Recall that in Section \ref{sec-end}, we calculated that the end Khovanov homology for the standard $\mathbb{R}^2$ in $\mathbb{R}^4$ is concentrated at homological degree $0$. Thus to show that the surface $\Sigma$ is not smoothly standard, all we need to do is to find a nonzero class for $\overleftharpoonup{Kh}(\Sigma)$ that is not in homological degree $0$.

        To find such a nonzero class, we will construct a homology class in $Kh(L_1)$ with bigrading $(h,q)=(-2,-4)$, and prove that it survives the limiting process, i.e., it is in the image of $Kh(\overline{W_2}, \overline{C_2})\circ\cdots\circ Kh(\overline{W_i}, \overline{C_i})$ for all $i>1$. This will give rise to a generator in the end Khovanov homology $\overleftharpoonup{Kh}^{-2,-3}(\Sigma)$, proving that $\Sigma$ is not smoothly equivalent to a standard plane in 4-space.

        \subsubsection{Defining a class in $Kh(L_1)$.}

            As a first step, we use another diagram for the surface $\Sigma$, as depicted in Figure \ref{20250504-8}. The key steps of the isotopy between them are shown in Figure \ref{20250504-9}.

            \begin{figure}[H]
                \centering
                \begin{picture}(17cm,5.57cm)    

                    \put(0cm,0cm){\includegraphics[width=17cm]{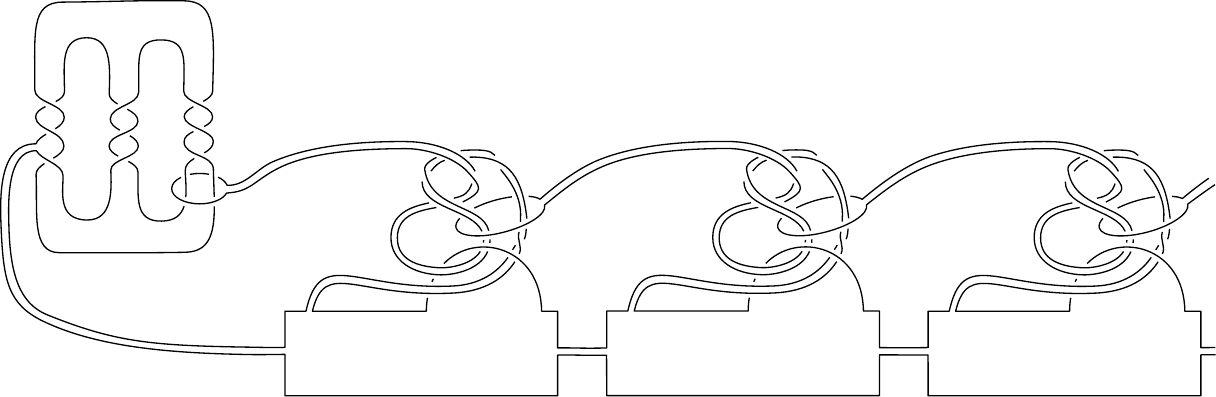}}
                    
                \end{picture}        
                \caption{An alternative diagram for the surface $\Sigma$.}
                \label{20250504-8}
            \end{figure}

            \begin{figure}
                \centering
                \begin{picture}(13.8cm,11.6cm)    

                    \put(0cm,8.4cm){\includegraphics[height=3.2cm]{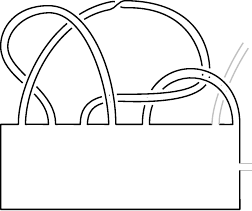}}
                    \put(5cm,8.4cm){\includegraphics[height=3.2cm]{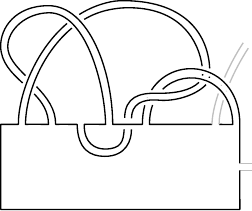}}
                    \put(10cm,8.4cm){\includegraphics[height=3.2cm]{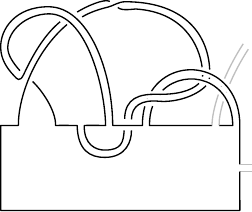}}

                    \put(0cm,4.2cm){\includegraphics[height=3.2cm]{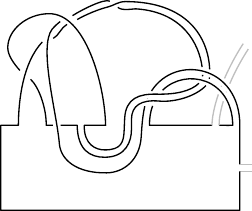}}
                    \put(5cm,4.2cm){\includegraphics[height=3.2cm]{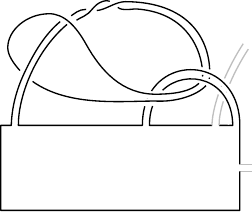}}
                    \put(10cm,4.2cm){\includegraphics[height=3.2cm]{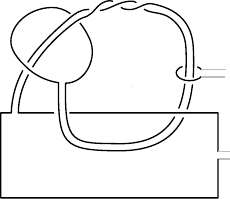}}

                    \put(0cm,0cm){\includegraphics[height=3.2cm]{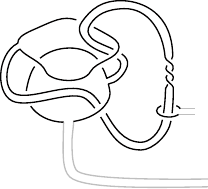}}
                    \put(5cm,0cm){\includegraphics[height=3.2cm]{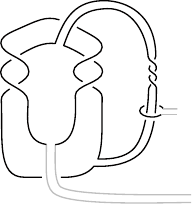}}
                    \put(10cm,0cm){\includegraphics[height=3.2cm]{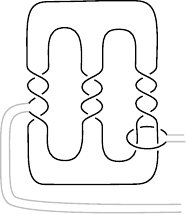}}

                \end{picture}        
                \caption{Key steps for the isotopy between Figure \ref{20250726-1} and Figure \ref{20250504-8}.\vspace{-5pt}}
                \label{20250504-9}
            \end{figure}

            Let $c_1\in CKh(L_1)$ be the cycle defined by the labeled smoothing depicted in Figure \ref{20250504-10}. In Figure \ref{20250504-10}, we highlight 0-smoothed crossings with green arcs. (We will omit these decorations in subsequent diagrams.) Note that this labeled smoothing represents a cycle because every 0-smoothing, when changed into a 1-smoothing, merges two $x$-labels (cf. Lemma \ref{when-cycle}).

            \begin{figure}[H]
                \centering
                \begin{picture}(8cm,4cm)
                    \put(0cm,0cm){\includegraphics[width=3cm, height=4cm]{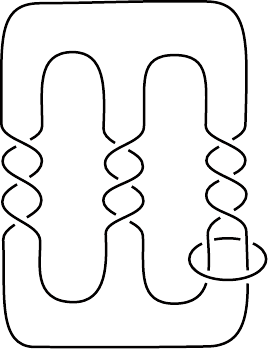}}

                    \put(5cm,0cm){\includegraphics[width=3cm, height=4cm]{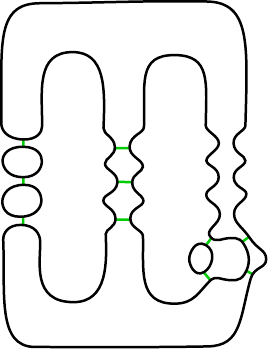}}

                    \put(6.3cm, 0.345cm){\scriptsize{$x$}}
                    \put(5.16cm, 2.075cm){\scriptsize{$x$}}
                    \put(6.88cm, 1.805cm){\scriptsize{$x$}}
                    \put(5.16cm, 1.605cm){\scriptsize{$x$}}
                    \put(7.17cm, 0.975cm){\scriptsize{$x$}}
                         
                \end{picture}        
                \caption{The cycle $c_1$.\vspace{-5pt}}
                \label{20250504-10}
            \end{figure}
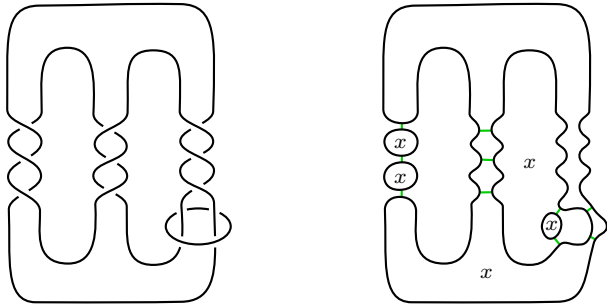

            Next we show that $c_1$ represents a nontrivial homology class in $Kh(L_1)$. To do so, we consider the image of $c_1$ under the map induced by the cobordism shown in Figure \ref{20250726-2}. However, the image of $c_1$ is a labeled smoothing that only contains 0-smoothings, so it cannot be a boundary and therefore represents a nontrivial homology class. It follows that $c_1$ is also a nontrivial homology class in $Kh(L_1)$.

            \begin{figure}
                \centering
                \begin{picture}(14.7cm,14.8cm)
                    \put(0cm,10.8cm){\includegraphics[width=3cm, height=4cm]{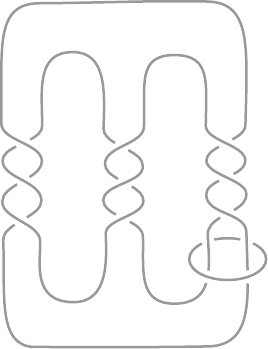}}
                    \put(4.5cm,10.8cm){\includegraphics[width=3cm, height=4cm]{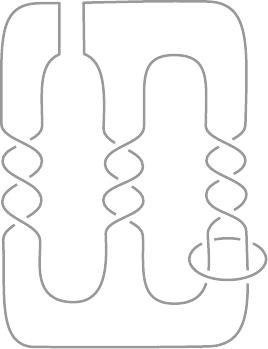}}
                    \put(9cm,10.8cm){\includegraphics[width=2.1cm, height=4cm]{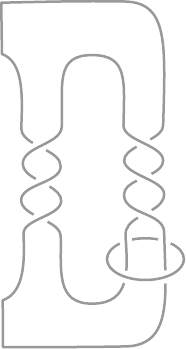}}
                    \put(12.6cm,10.8cm){\includegraphics[width=2.1cm, height=2.6cm]{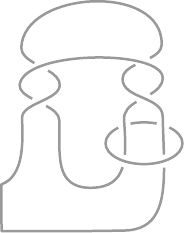}}

                    \put(0cm,6.2cm){\includegraphics[width=3cm, height=4cm]{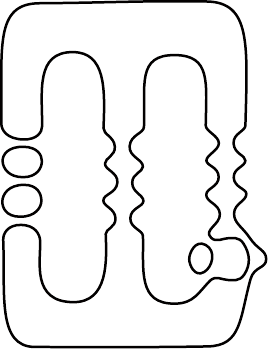}}
                    \put(4.5cm,6.2cm){\includegraphics[width=3cm, height=4cm]{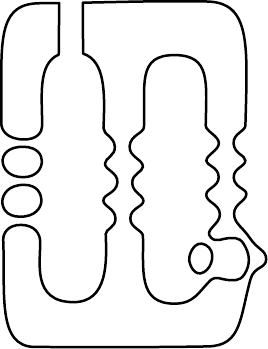}}
                    \put(9cm,6.2cm){\includegraphics[width=2.1cm, height=4cm]{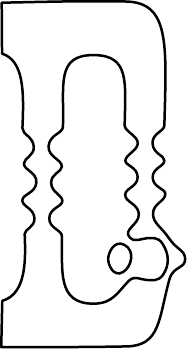}}
                    \put(12.6cm,6.2cm){\includegraphics[width=2.1cm, height=2.6cm]{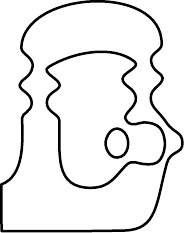}}

                    \put(-1.2cm, 5.4cm){\color{gray}\makebox[\linewidth]{\rule{14.7cm}{0.4pt}}}

                    \put(1.3cm, 6.545cm){\scriptsize{$x$}}
                    \put(0.16cm, 8.275cm){\scriptsize{$x$}}
                    \put(1.88cm, 8.005cm){\scriptsize{$x$}}
                    \put(0.16cm, 7.805cm){\scriptsize{$x$}}
                    \put(2.17cm, 7.175cm){\scriptsize{$x$}}

                    \put(5.8cm, 6.545cm){\scriptsize{$x$}}
                    \put(4.66cm, 8.275cm){\scriptsize{$x$}}
                    \put(6.38cm, 8.005cm){\scriptsize{$x$}}
                    \put(4.66cm, 7.805cm){\scriptsize{$x$}}
                    \put(6.67cm, 7.175cm){\scriptsize{$x$}}
                    \put(4.66cm, 9.30cm){\scriptsize{$x$}}

                    \put(9.4cm, 6.545cm){\scriptsize{$x$}}
                    \put(9.98cm, 8.005cm){\scriptsize{$x$}}
                    \put(10.27cm, 7.175cm){\scriptsize{$x$}}

                    \put(13.0cm, 6.545cm){\scriptsize{$x$}}
                    \put(13.58cm, 7.785cm){\scriptsize{$x$}}
                    \put(13.86cm, 7.145cm){\scriptsize{$x$}}

                    \put(0cm,2.7cm){\includegraphics[width=2.1cm, height=2.1cm]{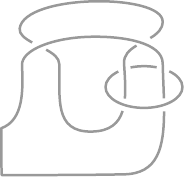}}
                    \put(4.5cm,2.7cm){\includegraphics[width=1.6cm, height=1.6cm]{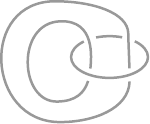}}

                    \put(0cm,0cm){\includegraphics[width=2.1cm, height=2.1cm]{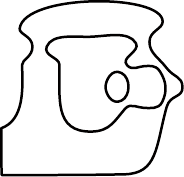}}
                    \put(4.5cm,0cm){\includegraphics[width=1.6cm, height=1.6cm]{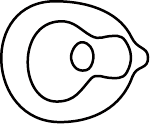}}

                    \put(0.4cm, 0.345cm){\scriptsize{$x$}}
                    \put(0.98cm, 1.385cm){\scriptsize{$x$}}
                    \put(1.26cm, 0.992cm){\scriptsize{$x$}}

                    \put(4.97cm, 0.792cm){\scriptsize{$x$}}
                    \put(5.305cm, 0.792cm){\scriptsize{$x$}}
                    \put(4.57cm, 0.792cm){\scriptsize{$x$}}

                \end{picture}        
                \caption{The cycle $c_1$ under the cobordism map.}
                \label{20250726-2}
            \end{figure}
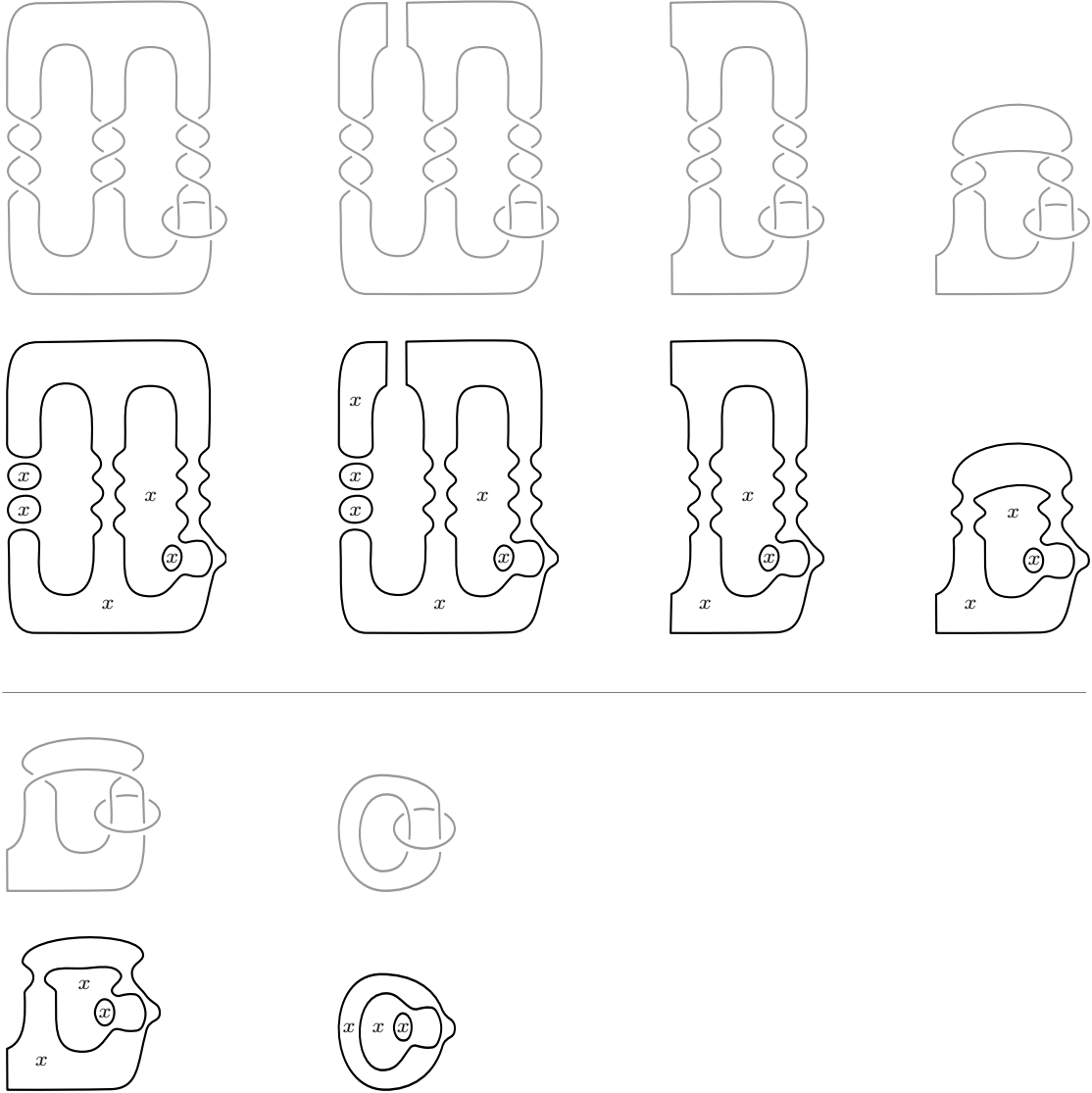
            
        \begin{remark}
            Note that the cobordism in the calculation above is different from the cobordism $(\overline{W}_1,\overline{C}_1)$, which maps the cycle $c_1$ to $0\in CKh(\varnothing)$.
        \end{remark}

        \subsubsection{It survives the limiting process.}

            Next we show that the homology class defined by the cycle $c_1$ survives the limiting process. More precisely, we show that for every integer $i>1$, there is a cycle $c_i\in CKh(L_i)$ such that $CKh(\overline{W_i},\overline{C_i})(c_i)=c_{i-1}$.

            Let $c_2$ be the cycle defined by the labeled smoothing of $L_2$ shown on the right side of Figure \ref{20250504-4-5} (as before, the labeled smoothing defines a cycle because every 0-smoothing, when changed into a 1-smoothing, merges two $x$-labels). Inductively, we define the cycle $c_i$ by building on the labeled smoothing $c_{i-1}$, adding in the same pattern, as shown in Figure \ref{20250828-1}.

            The calculation shown in Figures \ref{20250801-1} and \ref{20250801-2} illustrates that for all $i>1$, the cobordism-induced map $CKh(\overline{W_i}, \overline{C_i})$ maps $c_i$ to $c_{i-1}.$ This proves that the cycle $c_1$ survives the limiting process. Thus we conclude that the sequence $\{c_i\}_{i\in\mathbb{N}}$ defines a non-zero element in the end Khovanov homology $\overleftharpoonup{Kh}(\Sigma)$. 
            
            Now we claim that this element lies in the bigrading $(-2,-3)$ in $\overleftharpoonup{Kh}(\Sigma)$. Indeed, since $c_1\in CKh(L_1)$ has bigrading $(-2,-4)$, the bigrading $(h,q)$ of this element in the end Khovanov homology satisfies $(h,q+\chi(\Sigma\backslash\Sigma_1)) = (-2,-4)$. As $\chi(\Sigma_1)=2$, and $\chi(\Sigma\backslash\Sigma_1) = -1$, it follows that $\{c_i\}_{i\in\mathbb{N}}$ is a nontrivial class in $\overleftharpoonup{Kh}^{(-2,-3)}(\Sigma)$, which completes the proof that $\Sigma$ is not smoothly equivalent to a standard plane in 4-space.

            \begin{figure}
                \centering
                \begin{picture}(16.94cm,6.05cm)
                    \put(0cm,0cm){\includegraphics[width=7.97cm]{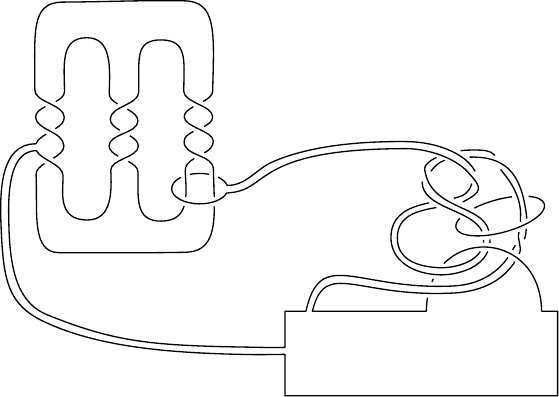}}
                    \put(8.97cm,0cm){\includegraphics[width=7.97cm]{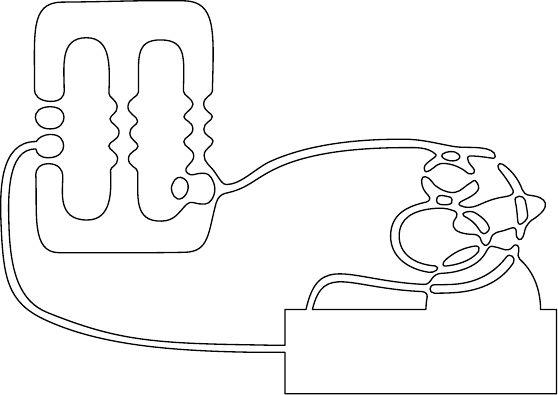}}

                    \put(14.08cm, 1.305cm){\tinytiny{$1$}}
                    \put(15.50cm, 1.905cm){\tinytiny{$1$}}
                    \put(15.71cm, 1.525cm){\tinyone{$x$}}
                    \put(15.265cm, 2.720cm){\tinyone{$x$}}
                    
                    \put(15.745cm, 2.425cm){\tinyone{$x$}}
                    \put(16.02cm, 2.555cm){\tinyone{$x$}}
                    \put(16.36cm, 2.61cm){\tinyone{$x$}}
                    
                    \put(15.35cm, 3.355cm){\tinyone{$x$}}     

                    \put(10.65cm, 2.345cm){\scriptsize{$x$}}
                    \put(9.61cm, 3.875cm){\scriptsize{$x$}}
                    \put(11.18cm, 3.705cm){\scriptsize{$x$}}
                    \put(9.61cm, 3.455cm){\scriptsize{$x$}}
                    \put(11.47cm, 2.875cm){\scriptsize{$x$}}

                \end{picture}        
                \caption{The link $L_2$ and the corresponding cycle $c_2$.\vspace{1.5cm}}
                \label{20250504-4-5}
            \end{figure}

            \begin{figure}
                \centering
                \begin{picture}(12.36cm,12.10cm)
                    \put(0cm,6.55cm){\includegraphics[width=12.36cm]{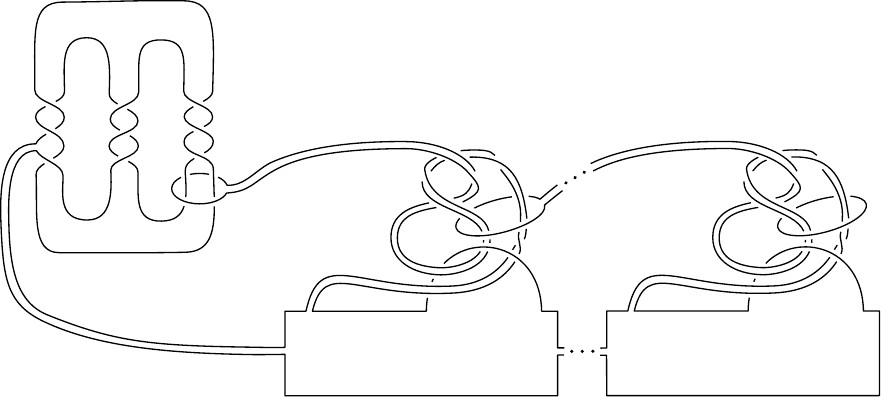}}
                    \put(0cm,0cm){\includegraphics[width=12.36cm]{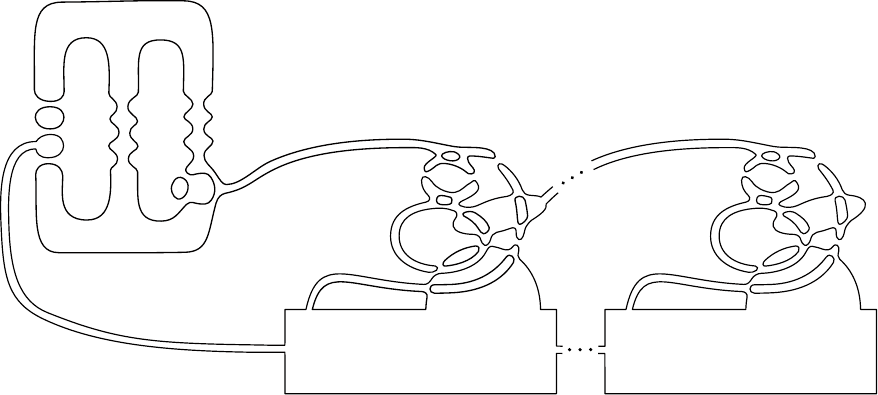}}

                    \put(5.11cm, 1.305cm){\tinytiny{$1$}}
                    \put(6.43cm, 1.875cm){\tinytiny{$1$}}
                    \put(6.64cm, 1.510cm){\tinyone{$x$}}
                    \put(6.195cm, 2.695cm){\tinyone{$x$}}                    
                    \put(6.665cm, 2.395cm){\tinyone{$x$}}
                    \put(6.95cm, 2.525cm){\tinyone{$x$}}
                    \put(7.275cm, 2.57cm){\tinyone{$x$}}              
                    \put(6.28cm, 3.31cm){\tinyone{$x$}}     

                    \put(1.66cm, 2.345cm){\scriptsize{$x$}}
                    \put(0.62cm, 3.850cm){\scriptsize{$x$}}
                    \put(2.21cm, 3.605cm){\scriptsize{$x$}}
                    \put(0.62cm, 3.440cm){\scriptsize{$x$}}
                    \put(2.47cm, 2.845cm){\scriptsize{$x$}}

                    \put(9.61cm, 1.305cm){\tinytiny{$1$}}
                    \put(10.93cm, 1.875cm){\tinytiny{$1$}}
                    \put(11.14cm, 1.510cm){\tinyone{$x$}}
                    \put(10.695cm, 2.695cm){\tinyone{$x$}}               
                    \put(11.165cm, 2.395cm){\tinyone{$x$}}
                    \put(11.45cm, 2.525cm){\tinyone{$x$}}
                    \put(11.775cm, 2.57cm){\tinyone{$x$}}              
                    \put(10.78cm, 3.31cm){\tinyone{$x$}}

                \end{picture}        
                \caption{The link $L_i$ and the corresponding cycle $c_i$.\vspace{1.5cm}}
                \label{20250828-1}
            \end{figure}

            \begin{figure}[H]
                \centering
                \begin{picture}(17cm,23.1cm)
                    \put(0cm,19.6cm){\includegraphics[height=3.5cm]{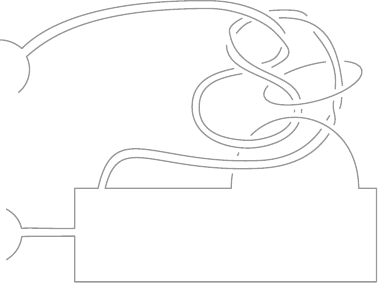}}
                    \put(6.2cm,19.6cm){\includegraphics[height=3.5cm]{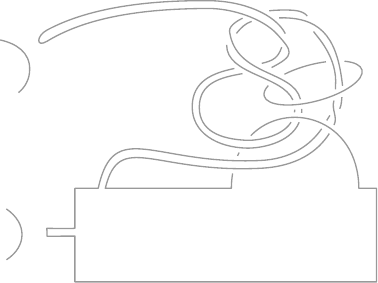}}
                    \put(12.4cm,19.6cm){\includegraphics[height=3.5cm]{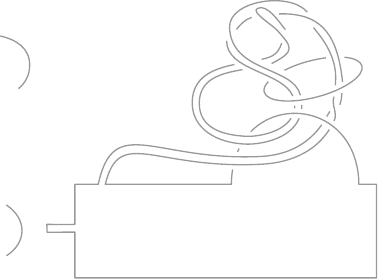}}

                    \put(0cm,15.8cm){\includegraphics[height=3.5cm]{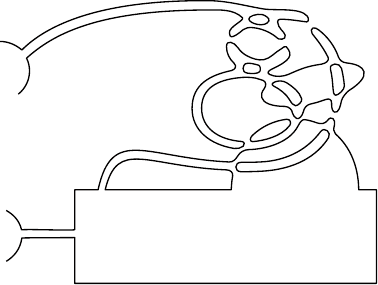}}
                    \put(6.2cm,15.8cm){\includegraphics[height=3.5cm]{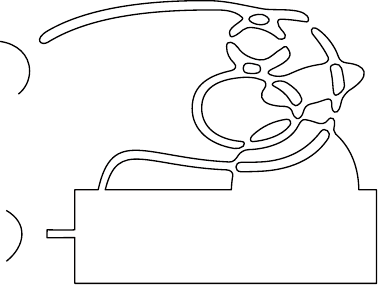}}
                    \put(12.4cm,15.8cm){\includegraphics[height=3.4cm]{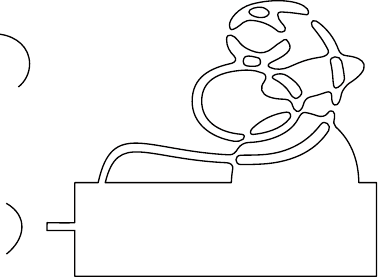}}

                    \put(0cm, 15.4cm){\color{gray}\makebox[\linewidth]{\rule{\textwidth}{0.4pt}}}
                    
                    \put(0.05cm, 16.355cm){\tinytiny{$x$}}
                    \put(0.05cm, 18.405cm){\tinytiny{$x$}}
                    \put(1.98cm, 17.055cm){\tinytiny{$1$}}
                    \put(3.27cm, 17.635cm){\tinytiny{$1$}}
                    \put(3.49cm, 17.275cm){\tinytiny{$x$}}
                    \put(3.02cm, 18.4115cm){\tinytiny{$x$}}
                    \put(3.500cm, 18.135cm){\tinytiny{$x$}}
                    \put(3.75cm, 18.205cm){\tinytiny{$x$}}
                    \put(4.095cm, 18.275cm){\tinytiny{$x$}}
                    \put(3.12cm, 19.015cm){\tinytiny{$x$}}

                    \put(6.25cm, 16.355cm){\tinytiny{$x$}}
                    \put(6.25cm, 18.405cm){\tinytiny{$x$}}
                    \put(8.18cm, 17.055cm){\tinytiny{$1$}}
                    \put(9.47cm, 17.635cm){\tinytiny{$1$}}
                    \put(9.69cm, 17.275cm){\tinytiny{$x$}}
                    \put(9.22cm, 18.4115cm){\tinytiny{$x$}}
                    \put(9.700cm, 18.135cm){\tinytiny{$x$}}
                    \put(9.95cm, 18.205cm){\tinytiny{$x$}}
                    \put(10.295cm, 18.275cm){\tinytiny{$x$}}
                    \put(9.32cm, 19.015cm){\tinytiny{$x$}}
                    \put(8.795cm, 16.355cm){\tinytiny{$x$}}
                    \put(9.50cm, 18.910cm){\tinytiny{$x$}}                  

                    \put(12.45cm, 16.355cm){\tinytiny{$x$}}
                    \put(12.45cm, 18.405cm){\tinytiny{$x$}}
                    \put(14.38cm, 17.055cm){\tinytiny{$1$}}
                    \put(15.67cm, 17.625cm){\tinytiny{$1$}}
                    \put(15.89cm, 17.275cm){\tinytiny{$x$}}
                    \put(15.42cm, 18.4115cm){\tinytiny{$x$}}
                    \put(15.895cm, 18.125cm){\tinytiny{$x$}}
                    \put(16.15cm, 18.205cm){\tinytiny{$x$}}
                    \put(16.480cm, 18.275cm){\tinytiny{$x$}}
                    \put(15.52cm, 19.010cm){\tinytiny{$x$}}
                    \put(14.995cm, 16.355cm){\tinytiny{$x$}}
                    \put(15.685cm, 18.910cm){\tinytiny{$x$}}

                    \put(0cm,11.6cm){\includegraphics[height=3.5cm, width=4.6cm]{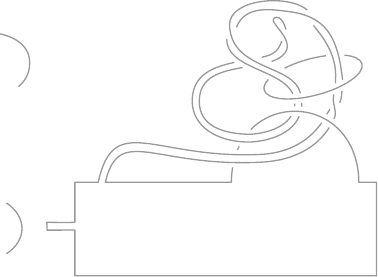}}
                    \put(6.2cm,11.6cm){\includegraphics[height=3.5cm, width=4.6cm]{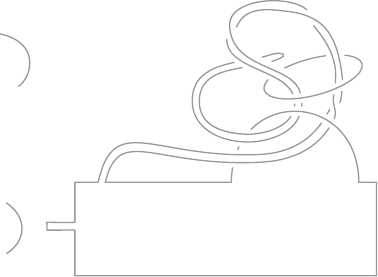}}
                    \put(12.4cm,11.6cm){\includegraphics[height=3.5cm, width=4.6cm]{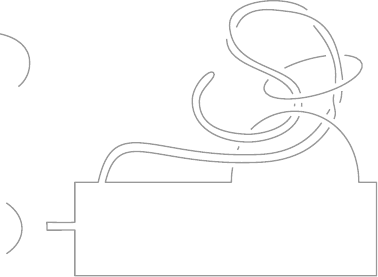}}

                    \put(0cm,7.9cm){\includegraphics[height=3.4cm]{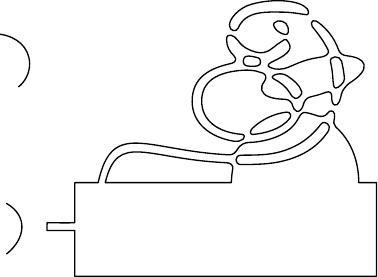}}
                    \put(6.2cm,7.9cm){\includegraphics[height=3.4cm]{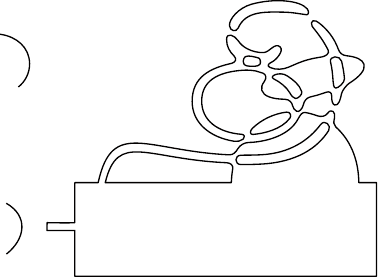}}
                    \put(12.4cm,7.9cm){\includegraphics[height=3.4cm]{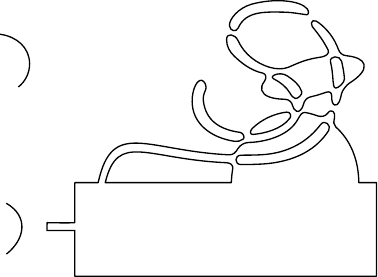}}

                    \put(0cm, 7.5cm){\color{gray}\makebox[\linewidth]{\rule{\textwidth}{0.4pt}}}

                    \put(0.05cm, 8.455cm){\tinytiny{$x$}}
                    \put(0.05cm, 10.505cm){\tinytiny{$x$}}
                    \put(1.98cm, 9.155cm){\tinytiny{$1$}}
                    \put(3.27cm, 9.730cm){\tinytiny{$1$}}
                    \put(3.49cm, 9.375cm){\tinytiny{$x$}}
                    \put(3.02cm, 10.5115cm){\tinytiny{$x$}}
                    \put(3.495cm, 10.225cm){\tinytiny{$x$}}
                    \put(3.75cm, 10.305cm){\tinytiny{$x$}}
                    \put(4.080cm, 10.375cm){\tinytiny{$x$}}
                    \put(3.22cm, 11.175cm){\tinytiny{$x$}}
                    \put(2.595cm, 8.455cm){\tinytiny{$x$}}
                    \put(3.395cm, 10.920cm){\tinytiny{$x$}}

                    \put(6.25cm, 8.455cm){\tinytiny{$x$}}
                    \put(6.25cm, 10.505cm){\tinytiny{$x$}}
                    \put(8.18cm, 9.155cm){\tinytiny{$1$}}
                    \put(9.47cm, 9.730cm){\tinytiny{$1$}}
                    \put(9.69cm, 9.375cm){\tinytiny{$x$}}
                    \put(9.22cm, 10.5115cm){\tinytiny{$x$}}
                    \put(9.695cm, 10.225cm){\tinytiny{$x$}}
                    \put(9.95cm, 10.305cm){\tinytiny{$x$}}
                    \put(10.280cm, 10.375cm){\tinytiny{$x$}}
                    \put(9.42cm, 11.175cm){\tinytiny{$x$}}
                    \put(8.795cm, 8.455cm){\tinytiny{$x$}}

                    \put(12.45cm, 8.455cm){\tinytiny{$x$}}
                    \put(12.45cm, 10.505cm){\tinytiny{$x$}}
                    \put(14.38cm, 9.155cm){\tinytiny{$1$}}
                    \put(15.67cm, 9.730cm){\tinytiny{$1$}}
                    \put(15.89cm, 9.375cm){\tinytiny{$x$}}
                    \put(14.795cm, 9.875cm){\tinytiny{$x$}}
                    \put(15.895cm, 10.225cm){\tinytiny{$x$}}
                    \put(16.15cm, 10.305cm){\tinytiny{$x$}}
                    \put(16.480cm, 10.375cm){\tinytiny{$x$}}
                    \put(15.62cm, 11.175cm){\tinytiny{$x$}}
                    \put(14.995cm, 8.455cm){\tinytiny{$x$}}

                    \put(0cm,3.7cm){\includegraphics[height=3.5cm, width=4.6cm]{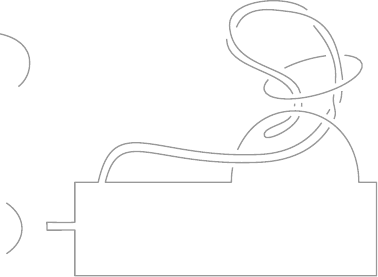}}
                    \put(6.2cm,3.7cm){\includegraphics[height=3.5cm, width=4.6cm]{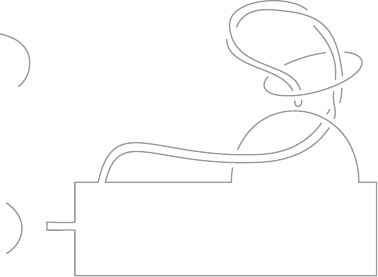}}
                    \put(12.4cm,3.7cm){\includegraphics[height=3.5cm, width=4.6cm]{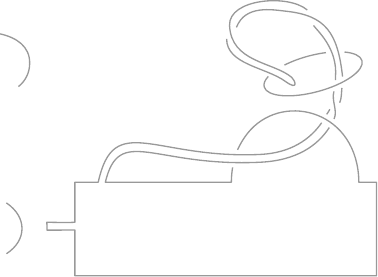}}

                    \put(0cm,0cm){\includegraphics[height=3.4cm]{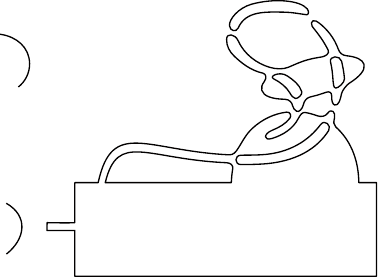}}
                    \put(6.2cm,0cm){\includegraphics[height=3.4cm]{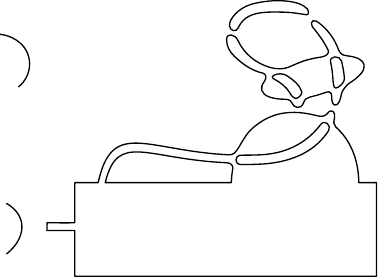}}
                    \put(12.4cm,0cm){\includegraphics[height=3.4cm]{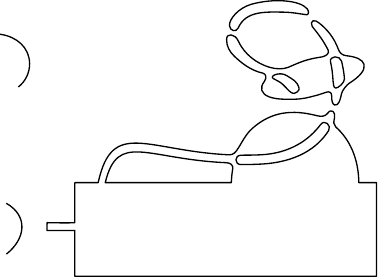}}

                    \put(0.05cm, 0.555cm){\tinytiny{$x$}}
                    \put(0.05cm, 2.605cm){\tinytiny{$x$}}
                    \put(1.98cm, 1.255cm){\tinytiny{$1$}}
                    \put(3.49cm, 1.475cm){\tinytiny{$x$}}
                    \put(3.495cm, 2.325cm){\tinytiny{$x$}}
                    \put(3.75cm, 2.405cm){\tinytiny{$x$}}
                    \put(4.080cm, 2.475cm){\tinytiny{$x$}}
                    \put(3.22cm, 3.275cm){\tinytiny{$x$}}
                    \put(2.595cm, 0.555cm){\tinytiny{$x$}}

                    \put(6.25cm, 0.555cm){\tinytiny{$x$}}
                    \put(6.25cm, 2.605cm){\tinytiny{$x$}}
                    \put(8.18cm, 1.255cm){\tinytiny{$1$}}
                    \put(9.69cm, 1.475cm){\tinytiny{$x$}}
                    \put(9.695cm, 2.325cm){\tinytiny{$x$}}
                    \put(9.95cm, 2.405cm){\tinytiny{$x$}}
                    \put(10.280cm, 2.475cm){\tinytiny{$x$}}
                    \put(9.42cm, 3.275cm){\tinytiny{$x$}}
                    \put(8.795cm, 0.555cm){\tinytiny{$x$}}

                    \put(12.45cm, 0.555cm){\tinytiny{$x$}}
                    \put(12.45cm, 2.605cm){\tinytiny{$x$}}
                    \put(14.38cm, 1.255cm){\tinytiny{$1$}}
                    \put(15.89cm, 1.475cm){\tinytiny{$x$}}
                    \put(15.895cm, 2.325cm){\tinytiny{$x$}}
                    \put(16.15cm, 2.405cm){\tinytiny{$x$}}
                    \put(16.480cm, 2.475cm){\tinytiny{$x$}}
                    \put(15.62cm, 3.275cm){\tinytiny{$x$}}
                    \put(14.995cm, 0.555cm){\tinytiny{$x$}}       
                    
                \end{picture}        
                \caption{The cobordism-induced map $CKh(\overline{W_i}, \overline{C_i})$ maps $c_i$ to $c_{i-1}.$}
                \label{20250801-1}
            \end{figure}

            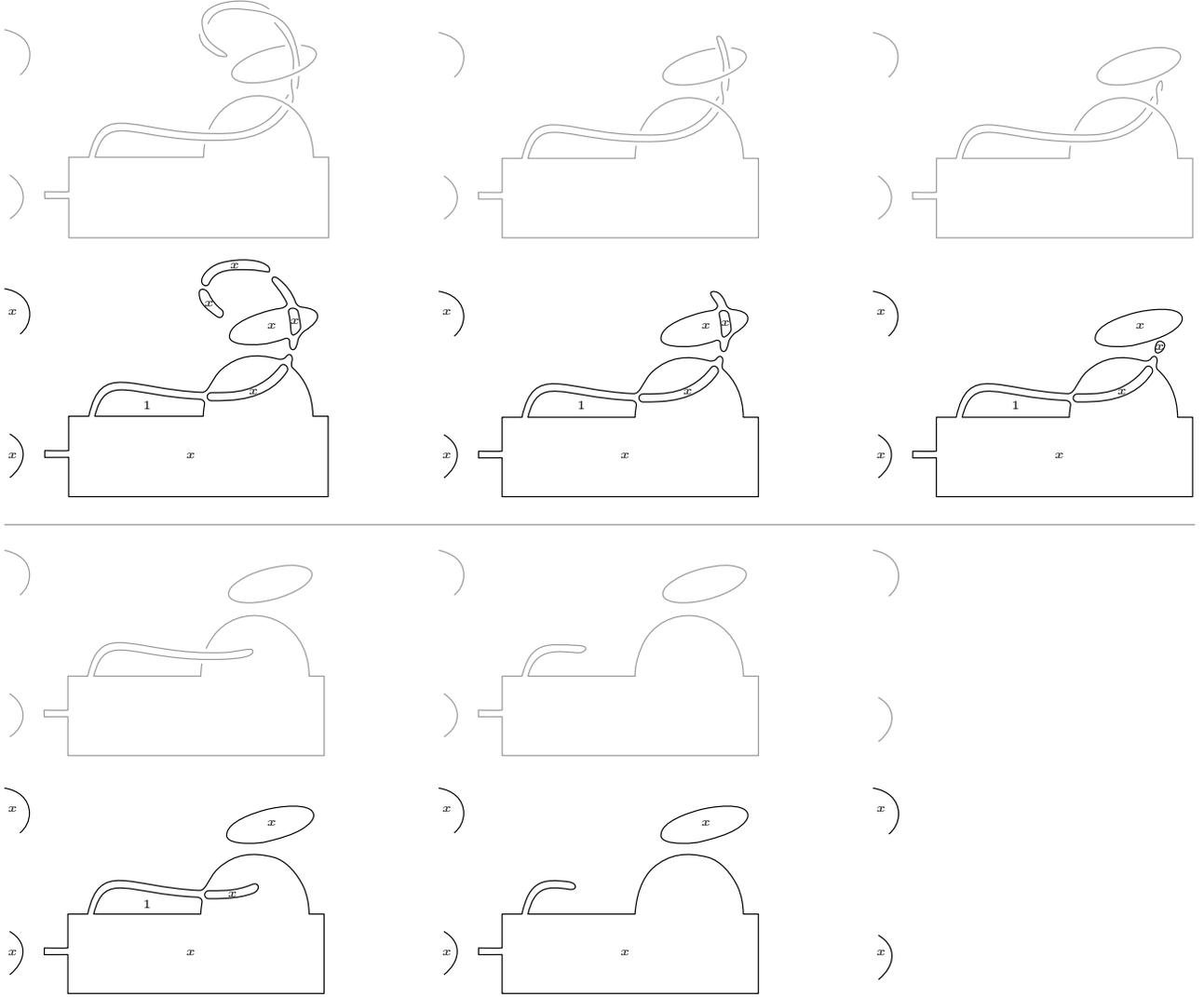
\begin{figure}[H]
                \centering
                \begin{picture}(17cm,14.2cm)
                    \put(0cm,10.8cm){\includegraphics[height=3.4cm]{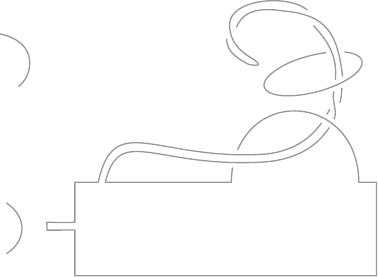}}
                    \put(6.2cm,10.8cm){\includegraphics[height=2.95cm]{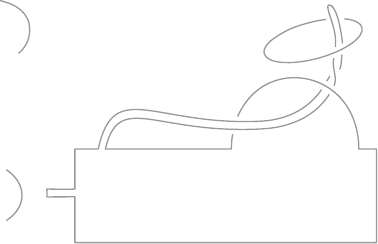}}
                    \put(12.4cm,10.8cm){\includegraphics[height=2.95cm]{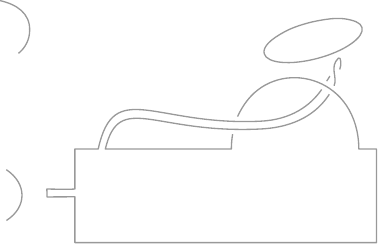}}

                    \put(0cm,7.1cm){\includegraphics[height=3.4cm]{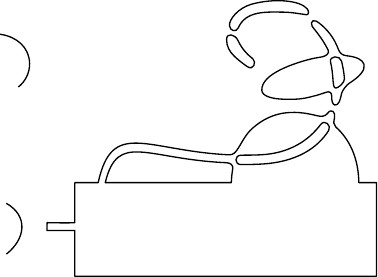}}
                    \put(6.2cm,7.1cm){\includegraphics[height=2.95cm]{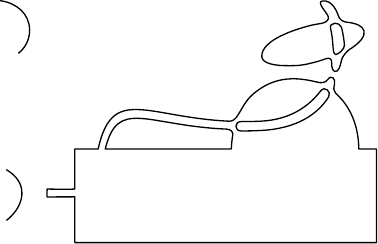}}
                    \put(12.4cm,7.1cm){\includegraphics[height=2.95cm]{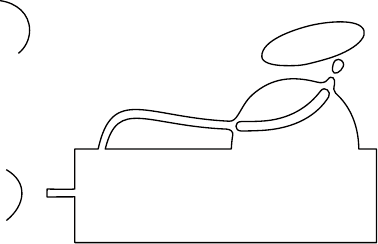}}

                    \put(0cm, 6.7cm){\color{gray}\makebox[\linewidth]{\rule{\textwidth}{0.4pt}}}

                    \put(0.05cm, 7.655cm){\tinytiny{$x$}}
                    \put(0.05cm, 9.705cm){\tinytiny{$x$}}
                    \put(1.98cm, 8.355cm){\tinytiny{$1$}}
                    \put(3.49cm, 8.575cm){\tinytiny{$x$}}
                    \put(2.855cm, 9.825cm){\tinytiny{$x$}}
                    \put(3.75cm, 9.505cm){\tinytiny{$x$}}
                    \put(4.080cm, 9.575cm){\tinytiny{$x$}}
                    \put(3.22cm, 10.375cm){\tinytiny{$x$}}
                    \put(2.595cm, 7.655cm){\tinytiny{$x$}}

                    \put(6.25cm, 7.655cm){\tinytiny{$x$}}
                    \put(6.25cm, 9.705cm){\tinytiny{$x$}}
                    \put(8.18cm, 8.355cm){\tinytiny{$1$}}
                    \put(9.69cm, 8.575cm){\tinytiny{$x$}}
                    \put(9.95cm, 9.505cm){\tinytiny{$x$}}
                    \put(10.225cm, 9.555cm){\tinytiny{$x$}}
                    \put(8.795cm, 7.655cm){\tinytiny{$x$}}

                    \put(12.45cm, 7.655cm){\tinytiny{$x$}}
                    \put(12.45cm, 9.705cm){\tinytiny{$x$}}
                    \put(14.38cm, 8.355cm){\tinytiny{$1$}}
                    \put(15.89cm, 8.575cm){\tinytiny{$x$}}
                    \put(16.15cm, 9.505cm){\tinytiny{$x$}}
                    \put(16.425cm, 9.205cm){\tinytiny{$x$}}
                    \put(14.995cm, 7.655cm){\tinytiny{$x$}}

                    \put(0cm,3.4cm){\includegraphics[height=2.95cm]{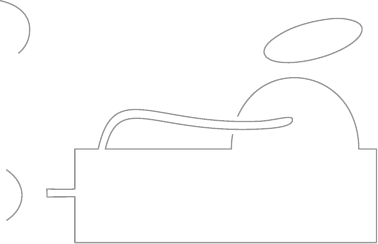}}
                    \put(6.2cm,3.4cm){\includegraphics[height=2.95cm]{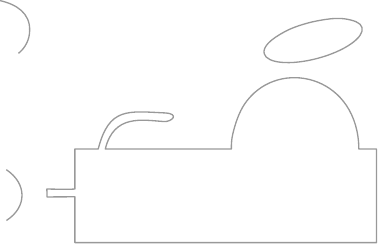}}
                    \put(12.4cm,3.6cm){\includegraphics[height=2.75cm]{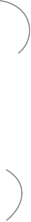}}

                    \put(0cm,0cm){\includegraphics[height=2.95cm]{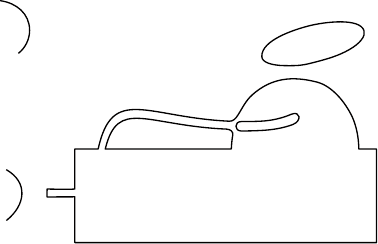}}
                    \put(6.2cm,0cm){\includegraphics[height=2.95cm]{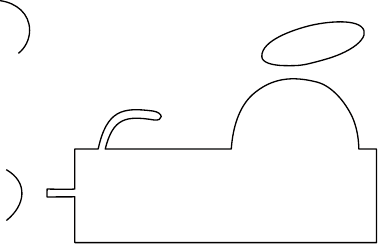}}
                    \put(12.4cm,0.2cm){\includegraphics[height=2.75cm]{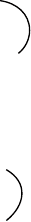}}

                    \put(0.05cm, 0.555cm){\tinytiny{$x$}}
                    \put(0.05cm, 2.605cm){\tinytiny{$x$}}
                    \put(1.98cm, 1.235cm){\tinytiny{$1$}}
                    \put(3.19cm, 1.390cm){\tinytiny{$x$}}
                    \put(3.75cm, 2.405cm){\tinytiny{$x$}}
                    \put(2.595cm, 0.555cm){\tinytiny{$x$}}

                    \put(6.25cm, 0.555cm){\tinytiny{$x$}}
                    \put(6.25cm, 2.605cm){\tinytiny{$x$}}
                    \put(9.95cm, 2.405cm){\tinytiny{$x$}}
                    \put(8.795cm, 0.555cm){\tinytiny{$x$}}

                    \put(12.45cm, 0.555cm){\tinytiny{$x$}}
                    \put(12.45cm, 2.605cm){\tinytiny{$x$}}

                \end{picture}        
                \caption{The cobordism-induced map $CKh(\overline{W_i}, \overline{C_i})$ maps $c_i$ to $c_{i-1}$ (continued).}
                \label{20250801-2}
            \end{figure}            

    \section{The Lagrangian and symplectic setting}{\label{sec-lagrangian}}

        Finally, we show that the surface $\Sigma$ can be smoothly isotoped into a Lagrangian submanifold of the standard symplectic $\mathbb{R}^4$.
        
        We work in the symplectization of the standard contact $S^3$, which is naturally symplectomorphic to the complement of the origin in $(\mathbb{R}^4,\omega_{\text{std}})$. Here we recall that the \textit{symplectization} of a contact 3-manifold $(Y,\xi=\operatorname{ker}(\alpha))$ is the 4-manifold $Y\times \mathbb{R}$ together with the symplectic form $\omega = \operatorname{d}(e^{t}\alpha).$ Furthermore, given a pair of numbers $a<b$, we have a \textit{truncated symplectization} given by the symplectic submanifold $Y\times [a,b]$.
        
        To prove that the surface $\Sigma$ can be smoothly isotoped into a Lagrangian submanifold, we need a modification of the following theorem.

        \begin{thrm}[\cite{BST15,Cha10,EHK16,Riz16}]{\label{thrm-Lag}}
            If two Legendrian links $L_-$ and $L_+$ in the standard contact $S^3$ are related by any of the following three moves, then there exists an exact, embedded, orientable, and collared Lagrangian cobordism from $L_-$ to $L_+$.
                \begin{itemize}
                    \item \textbf{Isotopy}: $L_-$ and $L_+$ are Legendrian isotopic.
                    \item \textbf{0-Handle addition}: The front of $L_+$ is the same as that of $L_-$ except for the addition of a disjoint Legendrian unknot as in the left side of Figure \ref{20250504-1}.
                    \item \textbf{1-Handle addition}: The fronts of $L_\pm$ are related by a Legendrian band move, i.e., the move shown on the right side of Figure \ref{20250504-1}.
                \end{itemize}

                \begin{figure}[H]
                    \includegraphics[width=7cm]{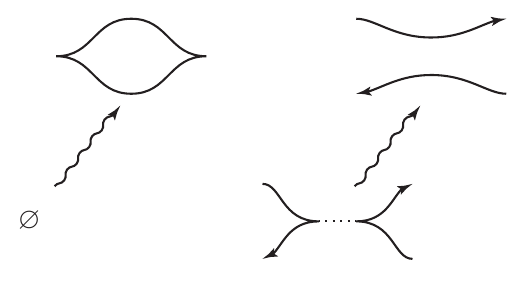}
                    \caption{Diagram moves corresponding to attaching a 0-handle and an oriented 1-handle. Picture from \cite{HS15}.}
                    \label{20250504-1}
                \end{figure}
                
        \end{thrm}

        The result is easily adapted to the noncompact setting of $(\mathbb{R}^4,\omega_{\text{std}})$. Here we let $(S_r, \xi_r)$ denote the 3-sphere of radius $r$ with its induced standard contact structure.

        \begin{cor}{\label{cor-Lag}}
            Let $\Sigma$ be a properly embedded surface in $\mathbb{R}^4$ described by a movie of link diagrams $L_r:=\Sigma\cap S_r$ (with $L_0=\varnothing$). If, for each positive integer $i$, 

            \begin{enumerate}
                \item the link $L_i$ is a Legendrian link in each $(S_i,\xi_i)$, and
                \item the link $L_{i}$ is obtained from $L_{i-1}$ by a finite sequence of Legendrian isotopies, 0-handle additions, and 1-handle additions as shown in Figure \ref{20250504-1},
            \end{enumerate}
            then $\Sigma$ is smoothly isotopic to a Lagrangian surface in $(\mathbb{R}^4,\omega_{\operatorname{std}})$.
            
        \end{cor}

        \begin{proof}
            The constructions underlying Theorem \ref{thrm-Lag} (from \cite{BST15,Cha10,EHK16,Riz16}) naturally yield a noncompact cobordism in $S^3\times \mathbb{R}$ whose ends are cylindrical outside of a compact subset $S^3\times [a,b].$ Therefore, for each integer $i\geq 1$, there exists a compact Lagrangian cobordism (with collared ends) between $L_{i-1}$ and $L_i$ in some compact subset $S^3\times [a_i,b_i]$.

            We wish to produce $\Sigma$ by stacking these compact cobordisms, but this requires translating them within $S^3\times \mathbb{R}$ to match up their ends. Although translation is not a symplectomorphism of the symplectization $(S^3\times \mathbb{R}, \operatorname{d}(e^t\alpha))$, it does preserve Lagrangian submanifolds and carry each level set $S^3\times \{t\}$ contactomorphically onto its image (via the natural identification). Therefore, after translations, we may assume $a_{i+1}=b_{i}$ for all positive integers $i$. Hence, we may stack the cobordisms between the links $L_i$. The resulting surface $\Sigma$ is the desired Lagrangian surface.
        \end{proof}

        For the remainder of this section, we apply this corollary to show that the surface $\Sigma$ described in Figure \ref{20250726-1} is smoothly isotopic to a Lagrangian surface.

        As a first step, we isotope the surface $\Sigma$ into the position described by Figure \ref{20250914-1}, which we denote by $\Sigma'$. In this position, its first-stage cobordism $C'_1$ is the obvious ribbon surface bounded by the Legendrian link described by the picture on the left of Figure \ref{20250914-2}. All later-stage cobordisms $C'_i$ are given by adding a trivial cobordism, followed by two 1-handle additions to the obvious ribbon surface bounded by the Legendrian link shown on the right of Figure \ref{20250914-2}. Also note that in this position, all of the ``boundary links" $L_i$ are already isotoped to Legendrian links, which we denote by $\mathcal{L}_i.$

        \begin{figure}
            \centering
            \begin{picture}(14.5cm,4cm)
                \put(0cm,0cm){\includegraphics[height=4cm]{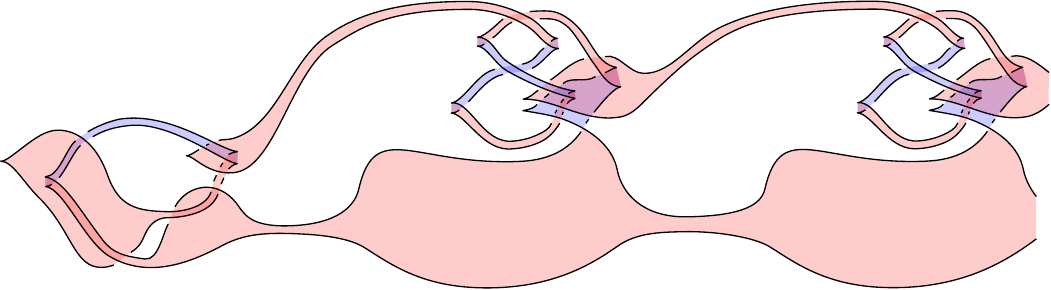}}
                
            \end{picture}        
            \caption{The isotoped exotic plane $\Sigma'$.}
            \label{20250914-1}
        \end{figure}

        \begin{figure}
                \centering
                \begin{picture}(10.5cm,3.5cm)      

                    \put(0cm,0cm){\includegraphics[height=2.5cm]{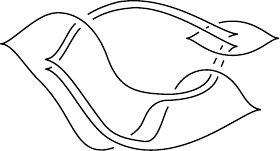}}
                    \put(6cm,0cm){\includegraphics[height=3.5cm]{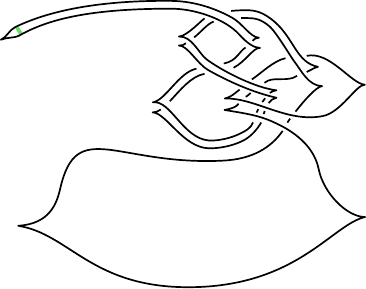}}

                \end{picture}        
                \caption{Left: the first-stage cobordism $C'_1$. Right: the ribbon surface that gives rise to later-stage cobordisms $C'_i$. }
                \label{20250914-2}
            \end{figure}

        Now in order to apply Corollary \ref{cor-Lag}, it suffices to check:
        \begin{enumerate}
            \item The surface $\Sigma'$ depicted in Figure \ref{20250914-1} is indeed smoothly isotopic to the original surface $\Sigma$.
            \item Each cobordism $C'_i$ can be realized as a Lagrangian cobordism from $\mathcal{L}_{i-1}$ to $\mathcal{L}_i$. That is, $\mathcal{L}_{i}$ can be obtained from $\mathcal{L}_{i-1}$ by a finite sequence of Legendrian isotopies and 0/1-handle additions as shown in Figure \ref{20250504-1}.
        \end{enumerate}

        \subsection*{\texorpdfstring{$\boldsymbol{\Sigma'}$}{Sigma'} is isotopic to \texorpdfstring{$\boldsymbol{\Sigma}$}{Sigma}}

            It suffices to check that all modified cobordisms $C'_i$ are smoothly isotopic to the original ones $C_i$. The key steps of the isotopy $C_1\sim C'_1$ are captured by Figure \ref{20250504-11}, and the key steps of the isotopy $C_i\sim C'_i$ for all $i>1$ are captured by Figure \ref{20250504-13}.

            Note that in the diagrams, we trace a set of bands, cutting along which results in an unlink in each diagram. Since the obvious disks bounded by these resulting unlinks are isotopic, it follows that the ribbon surfaces with the bands attached are also isotopic to each other. Hence, the surfaces $\Sigma$ and $\Sigma'$ are indeed isotopic.

            \begin{figure}[H]
                \centering
                \begin{picture}(15.2cm,7cm)      

                    \put(0cm,3.5cm){\includegraphics[height=3.5cm]{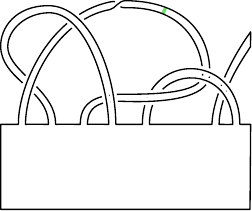}}
                    \put(5.5cm,3.5cm){\includegraphics[height=3.5cm]{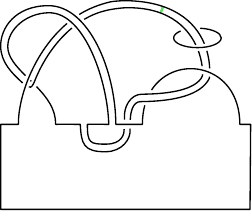}}
                    \put(11cm,3.5cm){\includegraphics[height=3.5cm]{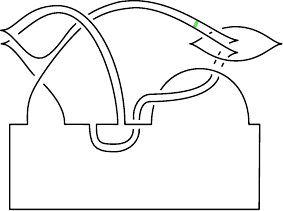}}

                    \put(0cm,0cm){\includegraphics[height=2.5cm]{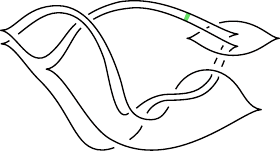}}
                    \put(5.5cm,0cm){\includegraphics[height=2.5cm]{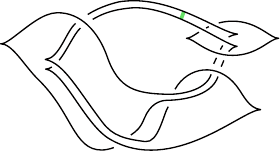}}

                \end{picture}
                \captionsetup{width = 15.2cm}
                \caption{A smooth isotopy between the link $L_1$ and its Legendrian representative $\mathcal{L}_1.$}
                \label{20250504-11}
            \end{figure}

            \begin{figure}
                \centering
                \begin{picture}(15.2cm,8cm)      

                    \put(0cm,4.5cm){\includegraphics[height=3.5cm]{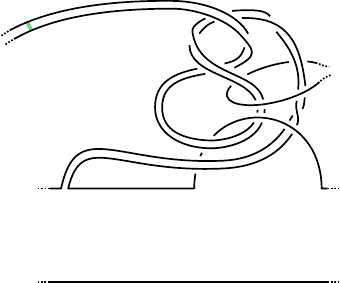}}
                    \put(5.5cm,4.5cm){\includegraphics[height=3.5cm]{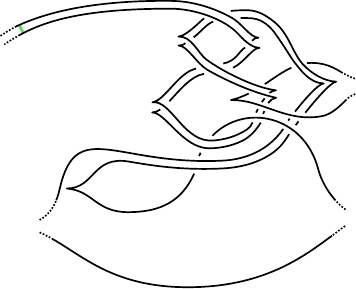}}
                    \put(11cm,4.5cm){\includegraphics[height=3.5cm]{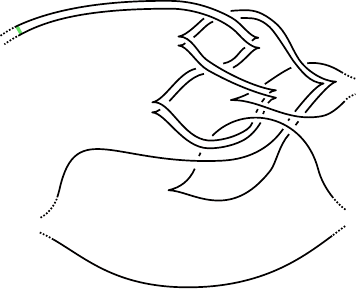}}

                    \put(0cm,0cm){\includegraphics[height=3.5cm]{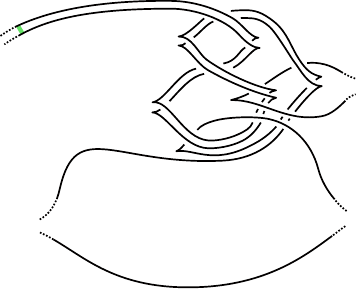}}
                    \put(5.5cm,0cm){\includegraphics[height=3.5cm]{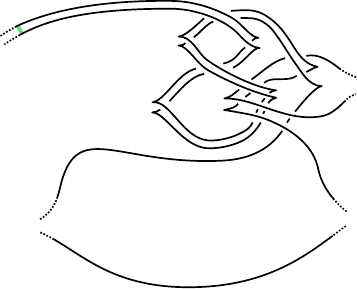}}

                \end{picture}
                \captionsetup{width = 15.2cm}
                \caption{A smooth isotopy between the link $L_i$ and its Legendrian representative $\mathcal{L}_i.$}
                \label{20250504-13}
            \end{figure}

        \subsection*{\texorpdfstring{$\boldsymbol{C'_i}$}{Ci} is a Lagrangian cobordism} Finally, we check that each cobordism $C'_i$ can be realized as a Lagrangian cobordism. In other words, we wish to obtain $\mathcal{L}_{i}$ from $\mathcal{L}_{i-1}$ via a finite sequence of Legendrian isotopies and 0/1-Lagrangian handle  additions.

        For the first-stage cobordism $C'_1$, the Legendrian link $\mathcal{L}_1$ is obtained from the empty link $\mathcal{L}_0$ by three 0-handle additions, one 1-handle addition, and Legendrian isotopies. The process is illustrated in Figure \ref{20250504-12}. 

            \begin{figure}
                \centering
                \begin{picture}(15.2cm,2.5cm)      

                    \put(0cm,0cm){\includegraphics[height=2.3cm]{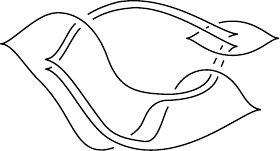}}
                    \put(5.5cm,0cm){\includegraphics[height=2.5cm]{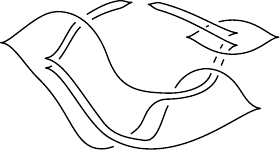}}
                    \put(11cm,0cm){\includegraphics[height=2.5cm]{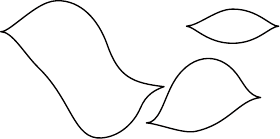}}

                \end{picture}        
                \caption{$\mathcal{L}_1$ is obtained from the empty link $\mathcal{L}_{0}$ by adding three Legendrian 0-handles and one Legendrian 1-handle.}
                \label{20250504-12}
            \end{figure}

            Next we check later-stage cobordisms $C'_i$ for $i>1$. By construction, each $\mathcal{L}_i$ is obtained from $\mathcal{L}_{i-1}$ by concatenating the pattern shown in the first picture of Figure \ref{20250504-14} via two Lagrangian 1-handle additions. On the other hand, the pattern is Legendrian isotopic to two Lagrangian 0-handles, as shown in Figure \ref{20250504-14}. It follows that $\mathcal{L}_i$ is obtained from $\mathcal{L}_{i-1}$ by two 1-handle additions, two 0-handle additions, and Legendrian isotopies.

            \begin{figure}
                \centering
                \begin{picture}(15.2cm,3.5cm)      

                    \put(0cm,0cm){\includegraphics[height=3.5cm]{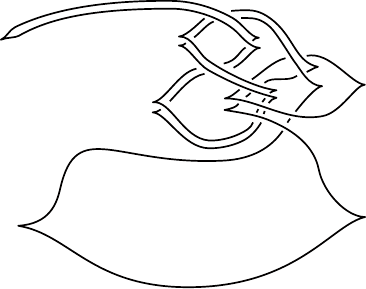}}
                    \put(5.5cm,0cm){\includegraphics[width=4.2cm]{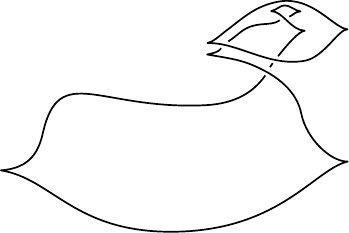}}
                    \put(11cm,0cm){\includegraphics[width=4.2cm]{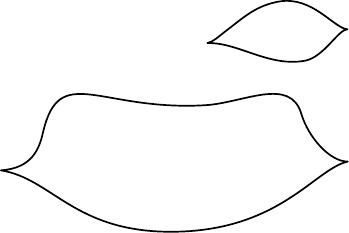}}

                \end{picture}        
                \caption{A Legendrian isotopy showing that $\mathcal{L}_i$ can be obtained from $\mathcal{L}_{i-1}$ by adding two Legendrian 0-handles and two Legendrian 1-handles.}
                \label{20250504-14}
            \end{figure}

        As each $\mathcal{L}_i$ is obtained from $\mathcal{L}_{i-1}$ via a finite sequence of Legendrian isotopies, 0-handle additions, and 1-handle additions, by Corollary \ref{cor-Lag}, we conclude that the surface $\Sigma'$ (and hence $\Sigma$) can be smoothly isotoped to a Lagrangian surface.

        This completes the proof of Theorem \ref{thrm-exist}.
        \hfill$\square$

        \subsection{The symplectic setting}

            Finally, we briefly discuss the symplectic setting. To isotope the surface into a symplectic submanifold, we need a symplectic/transverse analogue of Theorem \ref{thrm-Lag}, which we previously used to build Lagrangian cobordisms. The relevant theorem is stated below (cf. \cite{Hay21}, Lemma 3.2). For a proof of the theorem, we refer the readers to \cite{Hay21b} (Example 4.7 and Theorem 5.1) and \cite{EG22} (Section 2.2).
            
            \begin{thrm}{\label{thm-sym}}
                If two transverse links $L_-$ and $L_+$ in the standard contact $S^3$ are related by transverse isotopy and the moves depicted in Figure \ref{20250913-1}, then there exists an embedded, orientable, and collared symplectic cobordism from $L_-$ to $L_+$ in $S^3\times [a,b]$ for some choice of $a,b.$

                \begin{figure}[H]
                \centering
                \begin{picture}(7.5cm,3.75cm)  
                
                    \put(0.5cm,0cm){\includegraphics[width=7cm]{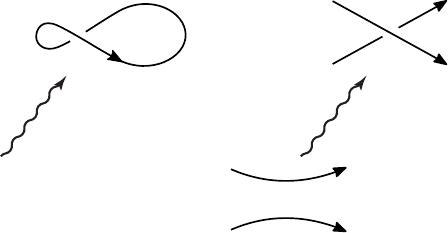} }
                    \put(0cm,0.5cm){$\varnothing$}
                
                \end{picture}        
                \caption{Transverse diagram moves corresponding to attaching a 0-handle and a 1-handle.}
                \label{20250913-1}
            \end{figure}
                
            \end{thrm}

            Analogously to its Lagrangian/Legendrian counterpart, this result can be easily adapted to the noncompact setting. 
            
            \begin{cor}{\label{cor-sym}}
                Let $\Sigma$ be a properly embedded surface in $\mathbb{R}^4$ described by a movie of link diagrams $L_r:=\Sigma\cap S_r$ (with $L_0=\varnothing$). If, for each positive integer $i$, 
    
                \begin{enumerate}
                    \item the link $L_i$ is a transverse link in each $(S_i,\xi_i)$, and
                    \item the link $L_{i}$ is obtained from $L_{i-1}$ by a finite sequence of transverse isotopies, 0-handle additions, and 1-handle additions as shown in Figure \ref{20250913-1},
                \end{enumerate}
                then $\Sigma$ is smoothly isotopic to a symplectic surface in $(\mathbb{R}^4,\omega_{\operatorname{std}})$.
            \end{cor}
            
            The proof is identical in structure to that of Corollary \ref{cor-Lag}, so we omit the details.

            On the other hand, every Legendrian link has a transverse push-off (cf. Section 2.9 of \cite{Etn05}). The existence of symplectic exotic planes follows from applying Corollary \ref{cor-sym} to the transverse push-off of the Legendrian links $\mathcal{L}_i\subset S_i$. Here we note that the 1-handle additions now occur at crossings corresponding to twists in the bands, instead of away from crossings as in the Lagrangian/Legendrian construction.

\appendix
\section{Abstract Khovanov Homology}

    In this appendix, we provide more detail on abstract Khovanov homology, mostly following the exposition in Section 4 of \cite{MWW22}. Our goal is to justify Lemma \ref{iso-natural} (restated as Lemma \ref{iso-natural2} later in this section).
            
    We first extend the definitions for Khovanov homology to links in 3-manifolds abstractly diffeomorphic to $\mathbb{R}^3$. The key is to take all possible diffeomorphisms into consideration.

    \begin{defn}
        Let $L$ be a link in a manifold $R$ abstractly diffeomorphic to $\mathbb{R}^3$. Then we define:
        \begin{itemize}
            \item $M(R)$ to be the space of diffeomorphisms $R\xrightarrow{\cong} \mathbb{R}^3$,
            \item $M(R,L)$ to be the subspace of $M(R)$ containing diffeomorphisms $\phi$ such that $\phi(L)$ is blackboard framed and generic when projecting to the $xy$-plane,
            \item $\pi_L: T(R,L)\rightarrow M(R,L)$ to be the fiber bundle whose fiber at the point $\phi$ is the standard Khovanov homology $Kh(\phi(L))$,
            \item the abstract Khovanov homology $Kh(R,L)$ to be the bigraded abelian group of flat sections on the bundle $T(R,L)$ (defined using a natural notion of parallel transport on $T(R,L)$). 
        \end{itemize}
    \end{defn}

    \begin{defn}{\label{def-cob-R}}
        Let $(W,\Sigma):(R_0,L_0)\rightarrow(R_1,L_1)$ be a link cobordism where $W$ is abstractly diffeomorphic to $\mathbb{R}^3\times I$. Furthermore, we pick a diffeomorphism $\phi: W\rightarrow \mathbb{R}^3\times I$ whose restrictions to the boundary $\phi_0:R_0\rightarrow \mathbb{R}^3$ and $\phi_1:R_1\rightarrow \mathbb{R}^3$ are in $M(R_0,L_0)$ and $M(R_1,L_1)$, respectively. Then we can define the Khovanov homology cobordism map: for an abstract Khovanov homology class (i.e., flat section) $\eta\in Kh(R_0,L_0)$, we define $Kh(W,\Sigma)(\eta)$ to be the unique flat section of $Kh(R_1,L_1)$ such that $Kh(W,\Sigma)(\eta)(\phi_1) = Kh(\phi(\Sigma))(\eta(\phi_0)).$
    \end{defn}

    It is shown in \cite{MWW22} that the cobordism map in Definition \ref{def-cob-R} is independent of the choice of the diffeomorphism $\phi$.

    Now we can define Khovanov homology for links in manifolds abstractly diffeomorphic to $S^3$.  

            \begin{defn}
                Let $L$ be a link in a manifold $S$ abstractly diffeomorphic to $S^3$. We define:

                \begin{itemize}
                    \item $\pi: T(S,L)\rightarrow S\backslash L$ be the fiber bundle whose fiber at a point $p$ is the bigraded abelian group $Kh(S\backslash\{p\},L)$.
                    \item $Kh(S,L)$ to be the bigraded abelian group of flat sections on $T(S,L).$
                \end{itemize}
            \end{defn}

            \begin{defn}{\label{def-cob-S}}
                Let $(W,\Sigma):(S_0,L_0)\rightarrow (S_1,L_1)$ be a link cobordism where $W$ abstractly diffeomorphic to $S^3\times I$. Furthermore, we pick a path $p_t\subset W\backslash\Sigma$ connecting $p_0\in S_0\backslash L_0$ and $p_1\in S_1\backslash L_1$.  Then we can define the Khovanov homology cobordism map as follows: for an abstract Khovanov homology class (i.e., flat section) $\eta\in Kh(S_0,L_0)$, we define $Kh(W,\Sigma)(\eta)$ to be the unique flat section of $Kh(S_1,L_1)$ such that $Kh(W,\Sigma)(\eta)(p_1) = Kh(W\backslash p_t, \Sigma)(\eta(\phi_0)).$
            \end{defn}

            Again, it is shown in \cite{MWW22} that the cobordism map in Definition \ref{def-cob-S} is independent of the choice of the path $p_t$.

            The main result of \cite{MWW22} is the following theorem.

            \begin{thrm}[\cite{MWW22}]
                Let $\BiAb$ be the category of bigraded vector spaces over $\mathbb{Z}/2$. The Khovanov homology for links in $\mathbb{R}^3$ generalizes to a functor $Kh_{\mathbb{R}^3}: \Cob^{iso}_{\mathbb{R}^3}\rightarrow \BiAb$, where 
                $$\Cob^{iso}_{\mathbb{R}^3}:= \begin{Bmatrix}
                    \text{link embeddings in oriented $R\cong \mathbb{R}^3$}\\
                    \text{link cobordisms in oriented $W\cong \mathbb{R}^3\times I$ up to isotopy rel boundary}
                \end{Bmatrix}.$$

                The Khovanov homology for links in $S^3$ generalizes to a functor $Kh_{S^3}: \Cob^{iso}_{S^3}\rightarrow \BiAb$, where 
                $$\Cob^{iso}_{S^3}:= \begin{Bmatrix}
                    \text{link embeddings in oriented $S\cong S^3$}\\
                    \text{link cobordisms in oriented $W\cong S^3\times I$ up to isotopy rel boundary}
                \end{Bmatrix}.$$
            \end{thrm}

            However, for our purpose to define the end Khovanov homology, the previous theorem is not quite enough. In particular, we need to modify the sources $\Cob^{iso}_{\mathbb{R}^3}$ and $\Cob^{iso}_{S^3}$ of the two functors, as shown in the theorem below. 

            \begin{thrm}{\label{quotient-category}}
                The Khovanov homology functor $Kh_{\mathbb{R}^3}:\Cob^{\text{iso}}_{\mathbb{R}^3}\rightarrow \BiAb$ factors through the quotient category $$\Cob^{\text{diff}}_{\mathbb{R}^3}:= \begin{Bmatrix}
                    \text{link embeddings in oriented $R\cong \mathbb{R}^3$}\\
                    \text{link cobordisms in oriented $W\cong \mathbb{R}^3\times I$ up to ambient diffeomorphisms rel boundary}
                \end{Bmatrix}.$$
            
                The Khovanov homology functor $Kh_{S^3}:\Cob^{\text{iso}}_{S^3}\rightarrow \BiAb$ factors through the quotient category $$\Cob^{\text{diff}}_{S^3}:= \begin{Bmatrix}
                    \text{link embeddings in oriented $S\cong S^3$}\\
                    \text{link cobordisms in oriented $W\cong S^3\times I$ up to ambient diffeomorphisms rel boundary}
                \end{Bmatrix}.$$
            \end{thrm}

            \begin{remark}
                This theorem is a direct consequence of the definitions of abstract Khovanov homology, and is implicitly mentioned in Section 5 of \cite{MWW22}. However, to the author's knowledge, it has not been spelled out explicitly in the literature.
            \end{remark}

            Spelling out the definitions of a quotient functor, Theorem \ref{quotient-category} becomes Lemmas \ref{auxilliary-R3} and \ref{auxilliary-S3}.

            \begin{lemma}{\label{auxilliary-R3}}
                Let $(W,\Sigma):(R_0,L_0)\rightarrow (R_1,L_1)$ be a cobordism of pairs where $W$ is abstractly diffeomorphic to $\mathbb{R}^3\times I$. Further assume that there is another 4-manifold $W'$ cobounded by $R_0$ and $R_1$, together with a diffeomorphism $F:W\rightarrow W'$ such that $F|_\partial =\id_{R_0\sqcup R_1}$. Then the two cobordisms $(W,\Sigma)$ and $(W', F(\Sigma))$ induce the same map $Kh(W,\Sigma)=Kh(W', F(\Sigma))$ between $Kh(R_0, L_0)$ and $Kh(R_1, L_1)$.
            \end{lemma}

            \begin{proof}
                We first pick a diffeomorphism $\phi:W\rightarrow\mathbb{R}^3\times I$. Then $\phi':=\phi\circ F^{-1}$ is a diffeomorphism $W'\rightarrow \mathbb{R}^3\times I$. Now consider any element $\eta\in Kh(R_0, L_0)$. By definition,
                
                \begin{itemize}
                    \item $Kh(W,\Sigma)(\eta)\in Kh(R_1, L_1)$ is the unique flat section that satisfies  $$Kh(W,\Sigma)(\eta)(\phi|_{R_1}) = Kh(\mathbb{R}^3\times I, \phi(\Sigma))(\eta(\phi|_{R_0})).$$
                    \item $Kh(W',F(\Sigma))(\eta)\in Kh(R_1, L_1)$ is the unique flat section that satisfies  $$Kh(W',F(\Sigma))(\eta)(\phi'|_{R_1}) = Kh(\mathbb{R}^3\times I, \phi'(F(\Sigma)))(\eta(\phi'|_{R_0})).$$
                \end{itemize}

                However, by construction we know that $\phi_{R_0} = \phi'_{R_0}$, $\phi_{R_1} = \phi'_{R_1}$, and $\phi(\Sigma) = \phi'(F(\Sigma))$. Thus it follows from uniqueness that the two induced maps are equivalent.
            \end{proof}

            \begin{lemma}{\label{auxilliary-S3}}
                Let $(W,\Sigma): (S_0,L_0)\rightarrow (S_1,L_1)$ be a link cobordism where $W$ is abstractly diffeomorphic to $S^3\times I$. Further assume that there is another 4-manifold $W'$ cobounded by $S_0$ and $S_1$, together with a diffeomorphism $F:W\rightarrow W'$ such that $F|_\partial =\id_{S_0\sqcup S_1}$, then the two cobordisms $(W,\Sigma)$ and $(W', F(\Sigma))$ induce the same map $Kh(W,\Sigma)=Kh(W', F(\Sigma))$ between $Kh(S_0, L_0)$ and $Kh(S_1, L_1)$.
            \end{lemma}

            \begin{proof}
                We first pick a path $p_t\subset W\backslash \Sigma$, whose endpoints are $p_0\in S_0\backslash L_0$ and $p_1\in S_1\backslash L_1$. Denote its image under $F$ by $p'_t$. Note that as $F$ restricts to the identity map on the boundary, we know that the endpoints of $p'_t$ are also $p_0$ and $p_1$.

                Let $\eta$ be an element in the $Kh(S_0, L_0)$. By definition, 
                
                \begin{itemize}
                    \item $Kh(W,\Sigma)(\eta)$ is the unique element in $Kh(S_1, L_1)$ such that $$Kh(W,\Sigma)(\eta)(p_1) = Kh(W\backslash p_t, \Sigma)(\eta(p_0)).$$
                    \item $Kh(W',F(\Sigma))(\eta)$ is the unique element in $Kh(S_1, L_1)$ such that $$Kh(W',F(\Sigma))(\eta)(p_1) = Kh(W'\backslash p'_t, F(\Sigma))(\eta(p_0)).$$
                \end{itemize} 
                
                On the other hand, according to the previous lemma, we know that the two right-hand sides are equal to each other (as $F$ restricts to a diffeomorphism $(W\backslash p_t, \Sigma)\rightarrow (W'\backslash p'_t, F(\Sigma))$ by construction). Thus, by uniqueness, we conclude that the two induced maps are equivalent.
            \end{proof}
            
            This completes the proof of Theorem \ref{quotient-category}. \hfill$\square$

            \medskip

            With Theorem \ref{quotient-category} in hand, it now makes sense to talk about isomorphisms induced by diffeomorphisms. Given a diffeomorphism $f:(S_0,L_0)\rightarrow(S_1,L_1)$, we can construct the 4-manifold $W:=S_0\times I\cup_f S_1\times I$ and the embedded surface $\Sigma:= L_0\times I\cup_f L_1\times I$. The cobordism $(W,\Sigma)$ induces a homomorphism $Kh(W,\Sigma): Kh(S_0,L_0)\rightarrow Kh(S_1,L_1).$

            \begin{lemma}
                The induced map $Kh(W,\Sigma): Kh(S_0,L_0)\rightarrow Kh(S_1,L_1)$ is an isomorphism.
            \end{lemma}

            \begin{proof}
                Define $\overline{W}$ as the 4-manifold $S_1\times I\cup_{f^{-1}} S_0\times I$, and $\overline{\Sigma}$ as the embedded surface $L_1\times I \cup_{f^{-1}} L_0\times I$. 

                We then observe that the pair $(W\cup_{\id} \overline{W}, \Sigma\cup_{\id}\overline{\Sigma})$ is diffeomorphic to the pair $(S_0\times I, L_0\times I)$ (via the diffeomorphism $\id\cup f^{-1}\cup \id$), and the pair $(\overline{W}\cup_{\id} {W}, \overline{\Sigma}\cup_{\id}\Sigma)$ is diffeomorphic to the pair $(S_1\times I, L_1\times I)$ (via the diffeomorphism $\id\cup f\cup \id$). This implies that the pairs $(W,\Sigma)$ and $(\overline{W}, \overline{\Sigma})$ are inverses of each other in the category $\Cob^{\text{diff}}_{S^3}$. Note that they are not inverses of each other in the category $\Cob^{\text{iso}}_{S^3}$.
                
                By Theorem \ref{quotient-category}, we conclude that the induced maps $Kh(W,\Sigma)$ and $Kh(\overline{W}, \overline{\Sigma})$ are inverses of each other as homogeneous morphisms between bigraded abelian groups. In particular, the induced maps are isomorphisms.
            \end{proof}

            Moreover, these induced isomorphisms are natural. Namely, we have the following lemma (cf. Lemma \ref{iso-natural}).

            \begin{lemma}{\label{iso-natural2}}
                Let $(W,\Sigma): (S_0,L_0)\rightarrow (S_1,L_1)$ and $(W',\Sigma'): (S'_0,L'_0)\rightarrow (S'_1,L'_1)$ be link cobordisms where $W,W'$ are abstractly diffeomorphic to $S^3\times I$. Then for any diffeomorphism of pairs $f:(W,\Sigma)\rightarrow(W',\Sigma')$, the following diagram commutes:
                \begin{center}
                    \begin{tikzcd}
                        {Kh(S_0,L_0)} \arrow[d, "f_\ast"] \arrow[rr, "{Kh(W,\Sigma)}"] &  & {Kh(S_1,L_1)} \arrow[d, "f_\ast"] \\
                        {Kh(S'_0,L'_0)} \arrow[rr, "{Kh(W',\Sigma')}"]                 &  & {Kh(S'_1,L'_1)}                   
                    \end{tikzcd}
                \end{center}
            \end{lemma}

            \begin{proof}
                Consider the two maps $Kh(W', \Sigma')$ and $f_\ast\circ Kh(W,\Sigma)\circ f^{-1}_\ast$ between $Kh(S'_0,L'_0)$ and $Kh(S'_1,L'_1)$:

                \begin{itemize}
                    \item The map $Kh(W', \Sigma')$ is induced by the cobordism $(W', \Sigma')$, which we rewrite as $$(S'_0\times I\cup_{\id} W'\cup_{\id} S'_1\times I, \,L'_0\times I\cup_{\id} \Sigma'\cup_{\id} L'_1\times I).$$
                    
                    \item The map $f_\ast\circ Kh(W,\Sigma)\circ f^{-1}_\ast$, according to the definition of induced isomorphisms, is induced by the cobordism $(S'_0\times I\cup_{f^{-1}} W \cup_{f} S'_1\times I,\, L'_0\times I\cup_{f^{-1}} \Sigma \cup_{f} L'_1\times I).$
                \end{itemize}

                However, the two cobordisms are diffeomorphic rel boundary. By Theorem \ref{quotient-category}, they induce the same map on Khovanov homology.
            \end{proof}

\medskip

\bibliographystyle{alpha} 
\bibliography{bib}

\begin{thebibliography}{MWW22}

\bibitem[AM75]{AM75}
S.S. Abhyankar and Tzuong-tsieng Moh.
\newblock Embeddings of the line in the plane.
\newblock {\em Journal für die reine und angewandte Mathematik}, 276:148--166, 1975.

\bibitem[BN05]{BN05}
Dror Bar-Natan.
\newblock Khovanov's homology for tangles and cobordisms.
\newblock {\em Geom. Topol.}, 9:1443--1499, 2005.

\bibitem[BST15]{BST15}
Fr\'ed\'eric Bourgeois, Joshua~M. Sabloff, and Lisa Traynor.
\newblock Lagrangian cobordisms via generating families: construction and geography.
\newblock {\em Algebr. Geom. Topol.}, 15(4):2439--2477, 2015.

\bibitem[Cha10]{Cha10}
Baptiste Chantraine.
\newblock Lagrangian concordance of {L}egendrian knots.
\newblock {\em Algebr. Geom. Topol.}, 10(1):63--85, 2010.

\bibitem[DR16]{Riz16}
Georgios Dimitroglou~Rizell.
\newblock Legendrian ambient surgery and {L}egendrian contact homology.
\newblock {\em J. Symplectic Geom.}, 14(3):811--901, 2016.

\bibitem[EG22]{EG22}
John~B. Etnyre and Marco Golla.
\newblock Symplectic hats.
\newblock {\em J. Topol.}, 15(4):2216--2269, 2022.

\bibitem[EHK16]{EHK16}
Tobias Ekholm, Ko~Honda, and Tam\'as K\'alm\'an.
\newblock Legendrian knots and exact {L}agrangian cobordisms.
\newblock {\em J. Eur. Math. Soc. (JEMS)}, 18(11):2627--2689, 2016.

\bibitem[Ell10]{Ell10}
Andrew Elliott.
\newblock {\em State cycles, quasipositive modification, and constructing {H}-thick knots in {K}hovanov homology}.
\newblock ProQuest LLC, Ann Arbor, MI, 2010.
\newblock Thesis (Ph.D.)--Rice University.

\bibitem[Etn05]{Etn05}
John~B. Etnyre.
\newblock Legendrian and transversal knots.
\newblock In {\em Handbook of knot theory}, pages 105--185. Elsevier B. V., Amsterdam, 2005.

\bibitem[FQ90]{FQ90}
Michael~H. Freedman and Frank Quinn.
\newblock {\em Topology of 4-manifolds}, volume~39 of {\em Princeton Mathematical Series}.
\newblock Princeton University Press, Princeton, NJ, 1990.

\bibitem[Gad10]{Gad10}
Siddhartha Gadgil.
\newblock Open manifolds, {O}zsv\'ath-{S}zab\'o{} invariants and exotic {$\mathbb R^4$}'s.
\newblock {\em Expo. Math.}, 28(3):254--261, 2010.

\bibitem[Gom84]{Gom84}
Robert~E. Gompf.
\newblock Infinite families of {C}asson handles and topological disks.
\newblock {\em Topology}, 23(4):395--400, 1984.

\bibitem[Gom25]{Gom25}
Robert~E. Gompf.
\newblock Topologically trivial proper 2-knots.
\newblock {\em Geom. Topol.}, 29(1):71--125, 2025.

\bibitem[GS99]{GS99}
Robert~E. Gompf and Andr\'{a}s~I. Stipsicz.
\newblock {\em {$4$}-manifolds and {K}irby calculus}, volume~20 of {\em Graduate Studies in Mathematics}.
\newblock American Mathematical Society, Providence, RI, 1999.

\bibitem[Hay21a]{Hay21c}
Kyle Hayden.
\newblock Corks, covers, and complex curves, 2021.

\bibitem[Hay21b]{Hay21}
Kyle Hayden.
\newblock Exotically knotted disks and complex curves, 2021.

\bibitem[Hay21c]{Hay21b}
Kyle Hayden.
\newblock Quasipositive links and {S}tein surfaces.
\newblock {\em Geom. Topol.}, 25(3):1441--1477, 2021.

\bibitem[HS15]{HS15}
Kyle Hayden and Joshua~M. Sabloff.
\newblock Positive knots and {L}agrangian fillability.
\newblock {\em Proc. Amer. Math. Soc.}, 143(4):1813--1821, 2015.

\bibitem[HS24]{HS24}
Kyle Hayden and Isaac Sundberg.
\newblock Khovanov homology and exotic surfaces in the 4-ball.
\newblock {\em J. Reine Angew. Math.}, 809:217--246, 2024.

\bibitem[Jac04]{Jac04}
Magnus Jacobsson.
\newblock An invariant of link cobordisms from {K}hovanov homology.
\newblock {\em Algebr. Geom. Topol.}, 4:1211--1251, 2004.

\bibitem[Kho00]{Kho00}
Mikhail Khovanov.
\newblock A categorification of the {J}ones polynomial.
\newblock {\em Duke Math. J.}, 101(3):359--426, 2000.

\bibitem[MWW22]{MWW22}
Scott Morrison, Kevin Walker, and Paul Wedrich.
\newblock Invariants of 4-manifolds from {K}hovanov-{R}ozansky link homology.
\newblock {\em Geom. Topol.}, 26(8):3367--3420, 2022.

\bibitem[OS05]{OS05}
Peter Ozsv\'ath and Zolt\'an Szab\'o.
\newblock On the {H}eegaard {F}loer homology of branched double-covers.
\newblock {\em Adv. Math.}, 194(1):1--33, 2005.

\bibitem[Rud82]{Rud82}
Lee Rudolph.
\newblock Embeddings of the line in the plane.
\newblock {\em Journal für die reine und angewandte Mathematik}, 337:113--118, 1982.

\bibitem[Ven97]{Ven97}
Gerard~A. Venema.
\newblock Local homotopy properties of topological embeddings in codimension two.
\newblock In {\em Geometric topology ({A}thens, {GA}, 1993)}, volume 2.1 of {\em AMS/IP Stud. Adv. Math.}, pages 388--405. Amer. Math. Soc., Providence, RI, 1997.

\end{thebibliography}

\end{spacing}
\end{document}